\documentclass[10pt]{article}

\usepackage[a4paper, total={5.5in, 10in}]{geometry}
\usepackage{amsfonts,amsmath,amscd,amssymb,stmaryrd,amsthm}
\usepackage{centernot}
\usepackage{cmap}
\usepackage[T1]{fontenc}
\usepackage[utf8]{inputenc}
\usepackage{lmodern}
\usepackage{graphicx,caption}
\usepackage{multirow}
\usepackage[english]{babel}
\usepackage{natbib}
\usepackage{xcolor}
\definecolor{color1}{gray}{1}
\usepackage{todonotes}
\usepackage[colorlinks=false]{hyperref}
\usepackage[capitalize]{cleveref}
\newcommand{\myref}[1]{\hyperref[#1]{\namecref{#1} \ref{#1}}}
\usepackage[normalem]{ulem}

\usepackage{tikz}
\usepackage{tikz-cd}
\usetikzlibrary{calc,positioning,arrows.meta,shapes.geometric,math}
\tikzset{
    buffer/.style={
        draw,
        shape border rotate=-90,
        isosceles triangle,
        isosceles triangle apex angle=60,
        node distance=2cm,
        minimum height=4em
    }
}
\tikzset{>={Latex[width=2mm,length=2mm]}}
\usepackage{ifthen}

\numberwithin{equation}{subsection}
\newtheorem{theorem}{Theorem}
\newtheorem{lemma}[theorem]{Lemma}
\newtheorem{proposition}[theorem]{Proposition}
\newtheorem{corollary}[theorem]{Corollary}
  
\newtheorem{definition}{Definition}
\newtheorem{example}{Example}
\newtheorem*{theorem*}{Theorem}
\newtheorem*{proposition*}{Proposition}
\newtheorem*{corollary*}{Corollary}
\newtheorem*{lemma*}{Lemma}
\theoremstyle{remark}

\newtheorem{question}{Question}
\newtheorem*{question*}{Question}
\newtheorem*{rmk*}{Remark}

\DeclareMathOperator{\id}{id}
\newcommand{\catname}[1]{\mathbf{#1}}

\newcommand{\st}{\, \mid \,}

\newcommand{\set}[1]{\{#1\}}
\newcommand{\sequence}[1]{\langle #1 \rangle}

\usepackage{mathtools}
\newcommand{\defeq}{\vcentcolon=}

\newcommand{\precut}[1]{\mathfrak{c}(#1)}
\newcommand{\cutalgebrasymb}{\mathfrak{C}}
\newcommand{\cutalgebra}[1]{\cutalgebrasymb(#1)}
\newcommand{\presep}[1]{\mathfrak{t}(#1)}
\newcommand{\separationalgebra}[1]{\mathfrak{T}(#1)}
\newcommand{\sepalgebrasymb}{\mathfrak{T}}
\renewcommand{\doteq}{\defeq} 

\renewcommand{\textbf}[1]{\textit{#1}}

\author{Violeta Mar \and Gustavo Boska} 

\newcommand{\Addresses}{{
  \bigskip
  \footnotesize
  G. ~Boska, \textsc{Instituto de Ci\^encias Matem\'aticas e de Computa\c c\~ao, Universidade de S\~ao Paulo\\
	Avenida Trabalhador s\~ao-carlense, 400,  S\~ao Carlos, SP, 13566-590, Brazil}\par\nopagebreak
  \textit{E-mail address}, G.~Boska: \texttt{gustavo.boska@usp.br}

  \medskip

  \footnotesize 
  V. ~Mar, \textsc{Instituto de Ci\^encias Matem\'aticas e de Computa\c c\~ao, Universidade de S\~ao Paulo\\
	Avenida Trabalhador s\~ao-carlense, 400,  S\~ao Carlos, SP, 13566-590, Brazil}\par\nopagebreak
  \textit{E-mail address}, G.~Mar: \texttt{violetamar@usp.br}
}}

\title{Describing ends and tangles (and their edge variants) through Boolean algebras and functors}
\begin{document}
    \newpage
    \maketitle
    \begin{abstract}
The end space of an infinite graph arises naturally in many contexts as an important invariant and an interesting construction. It compactifies a locally finite graph, and in \cite{tangle}, Diestel shows how to extend the end space to a larger space, called the tangle space, which is able to compactify any infinite graph. In both ends and tangles, it is the vertex-connectivity structure of the graph that is being studied. If we switch our attention to edge-connectivity, we can analogously define edge-ends. There is a space known as the edge-direction space which turns out to play an analogous role as the tangle space in its relationship with the end space: the edge-direction space provides a larger compact space in which the not necessarily compact space of edge-ends lives in. In this paper, we make this analogy precise, providing a natural edge analogue definition of tangles and proving they result in exactly the edge-directions. We also describe a combinatorial construction of certain Boolean algebras which give rise, via Stone duality, to the tangle and the edge-direction spaces. Finally, we pursue functorial definitions of the combinatorial constructions used in the paper, inspired by the famously functorial nature of Stone duality and by previous work by one of the authors and colleagues on trying to functorialize the end space construction. We hope our work will provide foundation and inspiration for further work on infinite graph theory that makes ample use of category theory and powerful algebraic constructions such as Boolean algebras and Stone duality.
    \end{abstract}

    \begin{keywords}
        graph,
        end space,
        edge-end space,
        direction space,
        edge-direction space,
        tangle,
        tangle space,
        edge-tangle,
        ultrafilter,
        boolean algebra,
        stone space,
        stone-duality,
        category theory,
        functor 
    \end{keywords}
    
    
    \section*{Introduction}

\paragraph{Ends, tangles and Boolean algebras}The end space of an infinite graph is one of the most ubiquitous constructions in infinite graph theory. It provides not only an important invariant associated with a given graph, but also a method for building interesting topological spaces. It compactifies the (geometric realization of the) graph when it is locally finite. In \cite{tangle}, Diestel shows how to extend the end space to a larger one, called the tangle space, which now is able to compactify (the geometric realization of) any infinite graph. In both ends and tangles, it is the vertex-connectivity structure of the graph that is being studied. If we switch our attention to edge-connectivity, we can analogously define edge-ends, which were explored and studied in \cite{lav}, \cite{real}, \cite{aurichi2024topologicalremarksendedgeend} and \cite{meninos2025}. In the vertex case, the notion of a direction is an equivalent definition of an end. When we turn to its edge analogue, though, the edge-direction space turns out to play an analogous role as the tangle space did for the end space, providing a larger compact space in which the not necessarily compact spaces of edge-ends will be in. In this paper, we make this analogy precise, providing a natural edge analogue definition of tangles and proving they result in exactly the edge-directions. These two spaces, the tangle space and the edge-direction space, were both known to be Stone spaces. The category of Stone spaces is famously equivalent to the category of Boolean algebras - this is the phenomenon known as Stone duality. We describe here a combinatorial construction of the Boolean algebra associated to each of these Stone spaces. This, in essence, provides a new construction of all of these spaces, which was previously only known and used in certain special cases by researchers working in geometric group theory, such as \cite{Dicks1980}. 

In quick summary: removing a finite set of vertices from a graph can result in a disconnected graph. These are usually called (finite) separations. In the same way,  removing a finite set of edges may also disconnect the graph, which in this case is called a (finite) cut. There are natural ways of giving to these combinatorial connectivity constructions the structure of a Boolean algebra. These will give rise to ultrafilter spaces through Stone duality - and it is those spaces which will have canonical relationships with the spaces already explored in the literature of ends, tangles, edge-ends and edge-directions. \\

\paragraph{Morphisms in graph theory} There are many concepts in graph theory that can be formulated in the categorical language of morphisms. For instance, the study of minors is a major subfield (and it is the place where the tangles we study here originally appeared). A minor of a graph $G$ is a graph $M$ obtained by removing vertices, edges or by contracting an edge into a single vertex. The operation of repeated edge contractions can be more generally described by a quotient. A quotient of a graph $G$ induced by a partition $P$ of $V(G)$ into pairwise disjoint subsets is a graph $G / P$ whose vertices are the set of equivalence classes $V(G) / P$ and two of those are connected by an edge when in $G$ there is a pair of adjacent vertices, one from each equivalent class. This construction can be captured by the projection map $\pi:G \to G/P$, which is a function from the set of vertices of $G$ to the set of vertices of $G / P$ such that adjacent vertices get mapped to either adjacent vertices or to a single vertex. A map with that property is called an 1-complex morphism. $\pi$ is also surjective - the existence of a surjective 1-complex morphism from a graph $G$ to a graph $Q$ is, then, equivalent to $Q$ being a quotient graph of $G$. If we further ask that the fibers $\pi^{-1}(v)$ for each $v \in V(Q)$ are connected, then $Q$ is a minor of $G$ obtained purely by edge contractions. The other operations, of removing vertices and edges, can be captured by an injective 1-complex morphism: $M$ is obtained from $G$ by removal of vertices or edges if and only if there is an injective 1-complex morphism $i: M \rightarrow G$. In summary, then, a graph $M$ is a minor of $G$ if and only if there is a graph $M'$ and two 1-complex morphisms $\pi:M' \rightarrow M$ and $i: M' \rightarrow G$, where $\pi$ is surjective with connected fibers and $i$ is injective.\footnote{Surjectivity and injectivity can be described categorically through the notion of terminal objects and point-surjectivity/injectivity - but it's not clear if we can describe `connected fibers' through categorical language.}\\

\paragraph{The end-space construction can be functorialized} In \cite{funcPaper}, one of the authors and colleagues study what kinds of graph morphisms give rise to a continuous map between end spaces - in essence, asking how to turn the end space construction into a functor. One such class, the \textbf{strongly proper morphisms}, is considered in this text. A possible application of this work is to use topological obstructions as arguments against the existence of quotient and minor relationships between graphs.

\paragraph{Some topological obstructions} For any topological space $X$, it is well defined its \textbf{Lindelöff number} $l(X)$, the least cardinal $\kappa$ such that any open covering of $X$ admits a sub-cover of cardinality $<\kappa$. For example, $l(X) = \aleph_0$ is equivalent to $X$ being compact. If there exists a surjective continuous map $f:X \to Y$ between two topological spaces, then $l(Y) \leq l(X)$. This means that given any pair of graphs $G_1,G_2$ with $l(\Omega(G_2)) > l(\Omega(G_1))$\footnote{See \cite{pitz} for a characterization of graphs (special trees) with arbitrarily large Lindelöf number $\kappa$.} and a strongly proper morphism $\phi:G_1 \to G_2$, one must have a $G_2$-ray $r_2$ such that for any $G_1$-ray $r_1$ one can separate $r_2$ from $\phi(r)$ by removing finitely many vertices. This condition is the consequence of $\Omega \phi: \Omega(G_1) \to \Omega(G_2)$ not being surjective. The same argument could be applied, under certain conditions, to the density, (local) weight and cellularity cardinal invariants. This could lead to interesting questions and interplay between combinatorial relationships between graphs and the topological and set-theoretical characteristics of their associated end spaces.

\paragraph{Our functorial goals} Inspired by those ideas, and by the use of the famously functorial Stone duality in constructing tangles and edge-directions (described in the first part of our paper), we pursue further the possibility of functorial definitions of the combinatorial constructions used in this paper. In summary, we ask if we can transform graph morphisms into morphisms between their corresponding line graphs (can we functorialize $L$?) and into Boolean algebra morphisms between their corresponding cut or separation algebras (can we functorialize $\cutalgebrasymb$ and $\sepalgebrasymb$?). And, if we can, can we use those morphisms to build continuous functions between their tangle spaces or their edge-direction spaces (can we functorialize $\Theta$ and $\mathcal{D}_E$?)?

    \subsection*{Preliminaries} 

\newcommand{\seq}[1]{\langle\, #1\, \rangle}

We collect here basic definitions and results we will use throughout the paper, giving references that expand on our brief account which can hopefully aid the unfamiliar reader.

\paragraph{Graph-theoretic conventions} We follow \cite{diestelgraphtheory} for the usual notations and conventions of basic graph theory. All graphs in this text shall be considered simple and undirected, so our edges are sets of two distinct vertices $e = \{v_1,v_2\}$, in which case we say $v_1$ and $v_2$ are \emph{incident} in $e$.

\paragraph{End spaces}
Given a graph $G$, a \emph{$G$-ray} is an infinite path of neighboring vertices $\langle v_1, v_2, \dotsc \rangle$ such that $\{v_i,v_{i+1}\} \in \mathrm{E}(G)$ for $i \in \mathbb{N}$ (by definition, a path doesn't repeat vertices, so a ray does not repeat vertices). Given $n\in\omega$, we say that $\seq{v_{n}, v_{n+1}, \dotsc}$ is a \emph{tail} of the $G$-ray $\seq{v_1, v_2, \dotsc}$. Two rays are said to be \emph{equivalent} if there exists infinitely many disjoint paths connecting both rays. This equivalence $r \sim s $ defines classes $[r]$ called \emph{ends}. We usually refer to $r$ as \emph{inducing} or \emph{representing} $[r]$. The set of $G$-ends is denoted by $\Omega(G)$. 

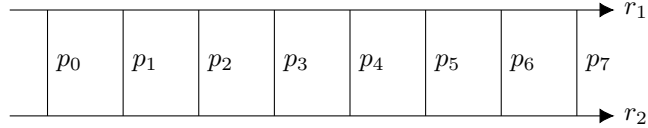
\begin{figure*}[ht]
\centering
\begin{tikzpicture}
    \def\radius{0.15}
    \def\dy{0.7}
    \def\dx{1}
    
    \draw[->] (-0.5*\dx,\dy)-- (7.5*\dx,\dy) node[right] {$r_1$};
    \draw[->] (-0.5*\dx,-\dy)-- (7.5*\dx,-\dy) node[right] {$r_2$};
    \foreach \i in {0,...,7}
    {
        \draw ({\i*\dx},\dy) -- ({\i*\dx},-\dy) node[midway,right] {$p_\i$};
    }
\end{tikzpicture}
\caption{$r_1$ and $r_2$ represent rays, and the $p_i$ represent either edges or paths between vertices of the rays. A graph like in this image is called a \emph{ladder}, and the two rays $r_1,r_2$ are called the ladder's \textbf{spines}. Two rays in a given graph $G$ are equivalent $s_1 \sim s_2$ exactly when there is a ladder $L \subset G$ with $s_1$ and $s_2$ as spines. }
\label{ladder}
\end{figure*}

There is a natural topology in $\Omega(G)$: given an end $\varepsilon\in \Omega (G)$ and a finite set of vertices $F\subset V(G)$, there exists a unique connected component $C(F,\varepsilon)$ of $G\setminus F$ that contains the tails of all representatives $r \in \varepsilon$. In this case, if we define  \[\Omega(F,\varepsilon)=\lbrace \eta \in \Omega (G) \,:\, C(F, \eta) = C(F, \varepsilon) \rbrace\subset \Omega(G) ,\]
then the collection $\lbrace \Omega (F, \varepsilon) : F\in [\mathrm{V}(G)]^{<\aleph_0}, \varepsilon\in \Omega (G)\rbrace$, where $[\mathrm{V}(G)]^{<\aleph_0}$ denotes the family of finite subsets of $\mathrm{V}(G)$, can be regarded as a basis of the (now topological) \emph{end space}. We say a vertex $v$ \emph{dominates a ray} $r$ if there are infinitely many vertex-disjoint paths from $v$ to $r$, and it \emph{dominates an end} when it dominates one of its representatives. We denote $\mathrm{S}_{G,F}$ as the set (or discrete space) whose elements are  the sets $V(C)$ for each infinite connected component $C$ of $G \backslash F$. 

\begin{figure}[ht]
\centering
\begin{tikzpicture}
    \def\radius{0.15}
    \def\dy{1}
    \def\dx{0.6}
    
        \draw[->] (-0.5*\dx,-\dy)-- (7.5*\dx,-\dy) node[right] {$r$};

        \foreach \i in {0,...,7}
        {
            \draw (3.5*\dx,0) -- ({\i*\dx},-\dy);
        }
        \node[shift = {(\dx,0)}] at (3.5*\dx,0) {$v$};

        \def \x {3.5}
        \foreach \n in {0,...,6}{
            \pgfmathtruncatemacro{\y}{6-\n+1}
            \draw (\x*\dx,0) -- (3.5*\dx,0);
            \draw (\x*\dx,0) -- (3.5*\dx,0);
        }
        
        \draw[fill = white] (3.5*\dx,0) circle (3pt);
\end{tikzpicture}
\caption{A vertex dominating a ray.}
\label{simplecomb}
\end{figure}
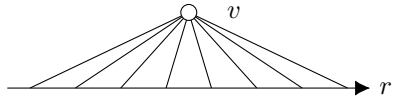

\paragraph{Edge-end spaces} An analogue of the usual end space in which \emph{edge-connectivity} rather than vertex-connectivity plays the central role is the notion of \emph{edge-end space}. This construction was introduced in \cite{lav} and has received only limited attention in the literature. We now recall its definition.

Two $G$-rays are said to be \textbf{edge-equivalent} if they cannot be separated by a finite set of edges. The equivalence classes under this relation are called the \textbf{edge-ends} of $G$. We denote the set of all edge-ends by $\Omega_E(G)$. In other words, for rays $r$ and $s$, we write $r \sim_E s$ whenever there is no finite set of edges $F \subset E(G)$ such that in $G \setminus F$ there are two different connected components, one with a tail of $r$ and another with a tail of $s$. The set $\Omega_E(G)$ is the set of equivalence classes of rays under this relation.

In close analogy with the construction of the usual end space, we endow $\Omega_E(G)$ with a natural topology. Let $\varepsilon \in \Omega_E(G)$ and let $F \subseteq E(G)$ be finite. There exists a unique connected component $C(F,\varepsilon)$ of $G \setminus F$ that contains a tail of every ray representing $\varepsilon$. We then define
\[
  \Omega_E(F,\varepsilon)
  =
  \{\, \eta \in \Omega_E(G) \st C(F,\eta)=C(F,\varepsilon) \,\}
\]
The family 
\(
  \{\, \Omega_E(F,\varepsilon) 
  \st F \in [E(G)]^{<\aleph_0},\ \varepsilon \in \Omega_E(G) \,\}
\)
forms a basis for a topology on $\Omega_E(G)$, now called the \textbf{edge-end space} of $G$.


\paragraph{Inverse systems} Recall (following \cite{eng}) that an \emph{inverse system} is a triple \[F \doteq (\{X_p\}_{p \in \mathbb{P}},\{\phi_{p_1p_2}:X_{p_2}\to X_{p_1}\}_{p_1 \leq p_2}, \mathbb{P})\] where $\mathbb{P}$ is an order, $X_p$ is a topological space for each $p \in \mathbb{P}$ and $\phi_{p_1p_2}$ is a continuous map, which is called a \emph{transition} or \emph{attaching} map. We ask $\phi_{pp}= \id_{X_p}$, that $\phi_{p_1p_2} \circ \phi_{p_2p_3} = \phi_{p_1p_3}$ and that $\mathbb{P}$ is an \emph{upward directed order}, meaning that for any pair $p_1,p_2 \in \mathbb{P}$ there is $p_3 \in \mathbb{P}$ such that $p_3 \geq p_1,p_2$. Define the \emph{limit} of this system as
\[ \varprojlim F = 
    \left\{
    (x_p)_{p \in \mathbb{P}} \in \prod_{p \in \mathbb{P}}X_p \st 
    \forall p_1 \leq p_2 
        (\phi_{p_1p_2}(x_{p_2}) = x_{p_1}) 
    \right\}
\]
taken with the subspace topology from the product. Note that the canonical projections $\pi_p:\varprojlim F \to X_p$ provide us a natural sub-basis of the inverse limit, as in $\{\pi_p^{-1}(A)\}_{p \in \mathbb{P},\, A \subset X_p \text{ open }}$. 

\paragraph{Directions} In \cite{DIESTEL2003197}, the following bijection is shown

\[
\Omega(G) \approx 
    \left\{
        \rho \in \prod_{F \in [\mathrm{V}(G)]^{<\aleph_0}} \mathrm{S}_{G,F} 
        \st 
        F \subset F' \in [\mathrm{V}(G)]^{< \aleph_0}
            \implies 
            \rho(F') \subset \rho(F)
    \right\}
\]

The ends, when viewed this way, are called \textit{directions}. It is routine to check that if we exchange $\mathrm{S}_{G,F}$ with \[\tilde{\mathrm{S}}_{G,F} \doteq \{C \in \mathrm{S}_{G,F} \st C \text{ has at least one ray}\}\] we get the same space. 

\paragraph{Star-Comb lemma} A \textbf{comb} is a pair $(r,\{p_n\})$, where $r$ is a ray and $\{p_n\}$ is an infinite collection of pairwise vertex disjoint paths. The ends of these paths $p_n$ are called the comb's \textbf{teeth}. An \textbf{infinite star subdivision}, or simply infinite star, consists of a vertex, called the \emph{center}, and infinitely many pairwise vertex-disjoint paths coming out of this center. The ends of these paths are called the \textbf{tips} of the star. The star-comb lemma claims it's impossible to forbid both of these types of graphs appearing in any connected infinite graph.

\begin{lemma}[Star-comb lemma]
    \label{star-comb}
    Let $U \subset \mathrm{V}(G)$ be an infinite vertex subset of a connected graph $G$. At least one of the following conditions must hold:
    \begin{itemize}
        \item[(i)] There exists a comb with infinitely many teeth in $U$.
        \item[(ii)] There exists a star with infinitely many tips in $U$.
    \end{itemize}
\end{lemma}

\paragraph{Boolean algebras} The archetypal example of a Boolean algebra is the lattice of subsets of a set $(2^X,\leq,\wedge,\vee,0,1,\neg)$ with order $\subset$, $\vee$-operation $\cup$, $\wedge$-operation $\cap$, minimum $\emptyset$, maximum $X$ and complement operation $\neg A \doteq X \backslash A$. For the general definition and standard results we will mention here, we refer the reader to \cite{boolean}. A \textbf{filter} of a Boolean algebra $\mathfrak{B}$ is an upward closed set $u \subset \mathfrak{B}$, such that $0_\mathfrak{B} \notin u$ and if $A,B \in u$ then $A \wedge B \in u$. An \textbf{ultrafilter} is a maximal filter. For each Boolean algebra $\mathfrak{B}$, $\mathrm{Ult}(\mathfrak{B}) $ is the set of \textbf{ultrafilters $u \subset \mathfrak{B}$}. We give this set a topology via the basic sets
\[
[B] \doteq \{u \in \mathrm{Ult}(\mathfrak{B}) \st B \in u\}
\]
for $B \in \mathfrak{B}$, giving rise to the \textbf{Stone space of this algebra}. A Stone space, in general, is a topological space that is compact, Hausdorff and zero-dimensional. The \textbf{clopen algebra} of a Stone space $X$ is the Boolean algebra of the clopen sets of this space, denoted by $\mathrm{Clopen}(X)$. Boolean morphisms $\phi:\mathfrak{B}_1 \to \mathfrak{B}_2$ give rise to a well defined function $\mathrm{Ult}(\phi):\mathrm{Ult}(\mathfrak{B}_2) \to \mathrm{Ult}(\mathfrak{B}_1)$, $\mathrm{Ult}(\phi)(u) \doteq \phi^{-1}[u]$.

We define $\catname{Bool}$ to be the category whose objects are Boolean algebras and morphisms are homomorphisms - functions from a Boolean algebra to another Boolean algebra that preserve union, intersection and complement. And $\catname{StoneSpace}$ is the category whose objects are Stone spaces and the morphisms are continuous functions.

\begin{theorem}[Stone Duality]
    There is a contra-variant categorical equivalence $\mathrm{Ult}:\catname{Boole} \to \catname{StoneSpace}$ whose inverse functor is $\mathrm{Clopen}:\catname{StoneSpace} \to \catname{Boole}$.
\end{theorem}

An ideal $\mathcal{I} \subset \mathfrak{B}$ is a `dual' notion of filter. Specifically, it is a subset of $\mathfrak{B}$ that is downward closed, $1_\mathfrak{B} \notin \mathcal{I}$ and $A\vee B \in \mathcal{I}$ whenever $A,B \in \mathcal{I}$. Now if we define $A \sim B$ when their symmetric difference \[A \Delta B \doteq (A \wedge \neg B) \vee (\neg A \wedge B) \in \mathcal{I},\] then $\sim$ is an equivalence relation and $\mathfrak{B}/\mathcal{I} \doteq \mathfrak{B}/\sim$ inherits the structure of a Boolean algebra. 

\begin{proposition}
    \label{quotient}
    The quotient $\pi_\mathcal{I}:\mathfrak{B} \to \mathfrak{B}/\mathcal{I}$ is a Boolean algebra homomorphism and is taken, through the ultrafilter functor, into a subspace inclusion \[\mathrm{Ult}(\pi_\mathcal{I}) = \iota_{\mathrm{Ult}(\mathfrak{B}/\mathcal{I})} \text{ where } \mathrm{Ult}(\mathfrak{B}/\mathcal{I}) \subset \mathrm{Ult}(\mathfrak{B})\] is the subspace of ultrafilters that do not contain any $I \in \mathcal{I}$.
\end{proposition}

To illustrate the above proposition,  consider the following example.

\begin{example}
    \label{quotientisremoval}
    Let $\mathcal{F} \subset 2^{X}$ be the ideal of finite sets of $X$. Then $\mathrm{Ult}(\mathcal{P}(X)/\mathcal{F})$ is the subspace of ultrafilters that do not select any finite set - therefore, these are the non-principal ultrafilters of $X$.
\end{example}




    \section{Tangles and the separation algebra}


\newcommand{\tangle}[1]{\Theta(#1)}
\newcommand{\almostequal}{\stackrel{*}{=}}

\subsection{The tangle space}

Tangles are, similarly to directions, an abstraction that tries to capture the concept of pointing towards `well connected' regions of a graph (i.e.: connected, $k$-connected, $k$-blocks and brambles). This will not be done by defining a specific region (a subset of vertices supposed to be the well connected region, for instance), but by a method of pointing towards a direction in which most of the well connected region is supposed to be in whenever the graph is disconnected by removing a certain number of vertices. If this way of pointing is `consistent' (in a sense to be defined), we call it a tangle.  Let's present the precise definition.

\paragraph{Tangles}

A \textit{separation} $\{A,B\}$ of a graph $G$ is composed of two vertex-sets $A,B \subset \mathrm{V}(G)$ with $\mathrm{V}(G) = A \cup B$ such that every edge is either in $G[A]$ or in $G[B]$. The \textit{separator} of a separation $\{A,B\}$ is $S \doteq A \cap B$, and the \textit{order} of a separation $\{A,B\}$ is the cardinal $\kappa = |S|$.

\begin{figure}[ht]
\centering
\begin{tikzpicture}[scale = 0.5]
    \draw (0,0) circle (2) node {$A$};
    \draw (3,0) circle (2) node {$B$};

    \draw[dashed] (0,1.7) -- (3,1.7) node[midway,above] {Forbidden edge};

    \node at (1.5,0) {$S$};
\end{tikzpicture}
\caption{Separation $(A,B)$ with separator $S$.}
\end{figure}
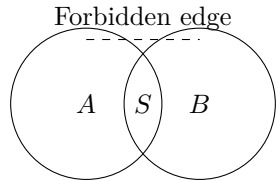 

An \textit{orientation of the (finite) separation} $\{A,B\}$ is an order choice \[\sigma \in \{(A,B),(B,A)\}\] for the separation. A $\kappa$\textit{-orientation} $\tau$ is a set that has exactly one orientation for each separation of order $<\kappa$. We intuitively interpret $(A,B)$ as an arrow pointing from $A$ to $B$, and we say $\tau$ picks $B$ when $(A,B) \in \tau$.

\begin{figure}[ht]
\centering
\begin{tikzpicture}[scale = 0.5]
    \draw[fill,teal] (3,0) circle (2.4);
    \node[teal,right] at (6,0) {$B$};
    \draw[fill,white] (0,0) circle (2.4); 
    \draw[teal] (3,0) circle (2.4);
    \node[magenta,left] at (0,0) {$A$};
    \draw[magenta] (0,0) circle (2.4);

    \draw[->] (0,2) -- (3,2) node[midway,above] {$\tau$};

    \node at (1.5,0) {$<\kappa$};
\end{tikzpicture}
\caption{An oriented $\kappa$-separation of a $\kappa$-orientation $(A,B) \in \tau$.}
\end{figure}
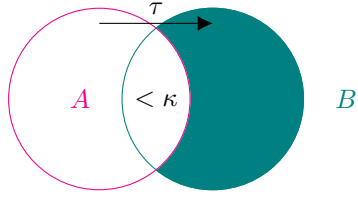

We say $n$ oriented separations $(A_1,B_1), \dots, (A_n, B_n)$ form a \textit{contradiction} if $G[A_1] \cup \dots \cup G[A_n] = G$. We say $n$ is the \textit{order} of this contradiction. A $k$-\textit{tangle} $\tau$ is a $k$-orientation with no contradictions of order $1$, $2$ or $3$. We define $\Theta_\kappa(G)$ as the set of $\kappa$-tangles of $G$ and if no $\kappa$ is specified, we will consider it $\Theta(G) \doteq \Theta_{\aleph_0}(G)$ as the \textbf{tangles of $G$}.

\begin{example}
   Let $G$ be a connected graph.
    \begin{itemize}
        \item[(i)] Let $v \in \mathrm{V}(G)$ and $\{ \{v\},\mathrm{V}(G)\}$ be a separation of order $1$, then $(\mathrm{V}(G),\{v\})\notin \tau$ for any 2-tangle.
        \item[(ii)] Let $S$ be an an infinite star with center $c$ and tips $n \in \mathbb{N}$. Take a separation of finite order $\{A,B\}$ where neither of the sets are the whole of $V(S)$.  It is routine to see that we must have $\{c \} \subset A \cap B$. Let $\mathcal{U}$ be an ultrafilter over $\mathbb{N}$, and define $\tau_{\mathcal{U}} \doteq \{(A,B) \st B \setminus \{c\} \in \mathcal{U}\}$. If we were to have a contradiction $(A_1,B_1), \dots,(A_n,B_n) \in \tau$ then $B_1 \cap B_2 \cap \cdots \cap B_n = \emptyset \in \mathcal{U}$, which is not possible. So, any ultrafilter over $\mathbb{N}$ gives rise to a unique tangle, and we can, reciprocally, show that any tangle arises uniquely from a ultrafilter in this way.
    \end{itemize}
\end{example}

Let's collect some lemmas from \cite{tangle} we will use.

\begin{lemma}[1.2 in \cite{tangle}]
\label{intersectionclosure}
    \sloppy Let $\tau$ be a $k$-tangle and $(A_1,B_1), (A_2,B_2) \in \tau$. If $(A_1\cup A_2) \cap B_1 \cap B_2$ has order $< k$, then $(A_1\cup A_2, B_1 \cap B_2) \in \tau$.
\end{lemma}
\begin{proof}
    Every vertex or edge that is not in the union of the induced subgraphs $G[A_1 ] \cup G[A_2]$ must be in $G[B_1 \cap B_2]$, which shows that $(A_1\cup A_2, B_1 \cap B_2)$ really does define a separation (of order $< k$ by hypothesis). It also shows that $(B_1 \cap B_2, A_1 \cup A_2), (A_1,B_1),(A_2,B_2)$ form a contradiction, and therefore $(B_1 \cap B_2, A_1 \cup A_2) \notin \tau$, which is the same as saying that $(A_1\cup A_2, B_1 \cap B_2) \in \tau$.
\end{proof}

\begin{lemma}[Lemma 1.10 of \cite{tangle}]
    \label{almostequalseparations}
    Let $\tau$ be a $\aleph_0$-tangle. If $(A,B) \in \tau$ and $A \sim A'$ and $B \sim B'$ (in other words, the symmetric difference of $A$ and $A'$, and the symmetric difference of $B$ and $B'$ are both finite), then $(A',B') \in \tau$.
\end{lemma}

\paragraph{Tangle spaces of graphs as Stone spaces} We follow \cite{tangle} to give the tangles of a graph a topological structure, which will turn out to be a Stone space. Consider the partial order $[\mathrm{V}(G)]^{< \aleph_0}$ of the finite vertex-sets, ordered by inclusion. For each set $F$, we assign the space $F \mapsto \Phi(F) \defeq \mathrm{Ult}(2^{\mathrm{S}_{G,F}})$, the space of ultrafilters of the set of infinite connected components of $G \backslash F$. Now, we specify the bonding maps of the inverse system of Stone spaces we want to construct. When $F \subset F'$ we define $f_{F,F'}:\Phi(F') \to \Phi(F)$ as
\[
u \in \mathrm{Ult}(2^{\mathrm{S}_{G,F'}}) \mapsto f_{F,F'}(u) \doteq \{C \in \mathrm{S}_{G,F} \st \exists C' \in u \, (C' \subset C)\}
\]
\begin{figure}[ht]
\centering
\begin{tikzpicture}[scale = 0.5,xscale = 1.6]
  \def\BigHalf{4.5}      
  \def\OrangeHalf{1.7}   
  \def\GreenHalf{0.75}   

  \def\CircleR{1.6}     
  \def\CircleCx{2.9}     
  \def\CircleCy{0.3}     

  \draw[line width=3pt, draw=black]
    (-0.8*\BigHalf,-0.8*\BigHalf) rectangle (1.2*\BigHalf,1.2*\BigHalf);

  \draw[line width=2pt, draw=orange!90!black]
    (-\OrangeHalf,-\OrangeHalf) rectangle (\OrangeHalf,\OrangeHalf) node[above,orange!90!black] {$F'$};

  \draw[line width=2pt, draw=green!45!black]
    (-\GreenHalf,-\GreenHalf) rectangle (\GreenHalf,\GreenHalf) node[above,green!45!black] {$F$};

  \draw[line width=2pt, draw=green!45!black]
    (\CircleCx,\CircleCy) circle (\CircleR);

\node[above,green!45!black] at (3.4,2) {$C \in f_{F,F'}(u)$};

  \foreach \p in {(3.35,1.20),(3.75,0.25),(3.30,-0.95)}{
    \draw[line width=1.8pt, draw=orange!90!black]
      \p circle (0.28);
  }

  \node[left,orange!90!black] at (3.5,0.25) {$C' \in u$};

\end{tikzpicture}
\caption{A schematic depiction of the attachment map.}
\end{figure}
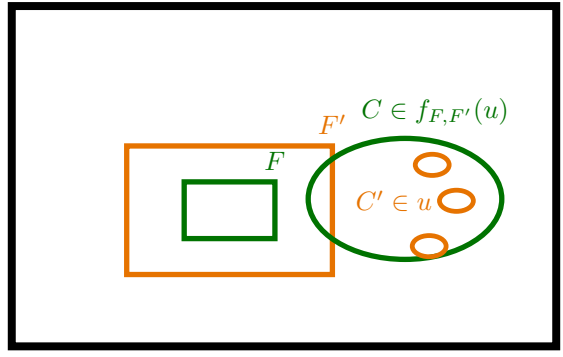

Now, \[\Phi_G \doteq \left( \{F \mapsto \mathrm{Ult}(2^{\mathrm{S}_{G,F}})\}_{F \in [\mathrm{V}(G)]^{< \aleph_0}},\{f_{F,F'}\}_{F \subset F' \in [\mathrm{V}(G)]^{< \aleph_0}}\right)\]

defines an inverse system. Its inverse limit has a natural bijection to the set of tangles of a graph.

\begin{theorem}[\cite{tangle}'s theorem 2]
    \label{inverselimitcharact}
    For any graph $G$, we have a canonical correspondence \[\Theta(G) \leftrightarrow \varprojlim \Phi_G.\]
\end{theorem}

Observe that this automatically gives $\Theta(G)$ the structure of a Stone space: by Tychonoff, the (non-empty) product $\prod_{F \in [\mathrm{V}(G)]^{<\aleph_0}}\mathrm{Ult}(2^{S_{G,F}})$ of compact Hausdorff spaces is a (non-empty) compact Hausdorff space. The inverse limit is a closed subset of the (compact Hausdorff) product, therefore is itself a compact Hausdorff space. Zero dimensionality is also inherited via the subspace relationship. 

The tangle space is a natural extension of the end space, which is only different from it in the non locally finite case. In summary

\begin{proposition}[\cite{tangle}]
    $\Theta(G)$ is a compact, zero dimensional Hausdorff space, i.e., a Stone space. There is a topological inclusion $\Omega(G) \rightarrow \Theta(G)$. This is a homeomorphism when $G$ is locally finite.
\end{proposition}

    \subsection{The separation algebra} As we just argued, $\Theta(G)$ is an inverse limit of Stone spaces, a Stone space itself. In the following, we find this Stone space's associated Boolean algebra in terms of the original graph and its connectivity properties.

\begin{definition}[Separation algebra]
    The set $\presep{G}  \defeq \{ E \subset \mathrm{E}(G) \mid \mathrm{V}(E) \cap \mathrm{V}(E(G) \setminus E) \text{ is finite} \}$ regarded as a Boolean algebra via the usual operations on $2^{E(G)}$ will be called $G$'s \textit{finite pre-separation algebra}, or simply \textit{pre-separation algebra}. Let $F_G \subset \separationalgebra{G}$ be the ideal of finite sets of edges, and define $\separationalgebra{G}\doteq \presep{G}/F_G$ as the \textit{separation algebra} of $G$. 
\end{definition}

\begin{rmk*}
    We will not use different notations for $A \in \presep{G}$ and $[A] \in \separationalgebra{G}$ for simplicity.
\end{rmk*}

Of course, one must verify that $\presep{G}$ is indeed closed under union, intersection and complement, which is a routine set theoretic calculation.

An element $E \in \separationalgebra{G}$ defines a finite order separation $\{\mathrm{V}(E), \mathrm{V}(G) \backslash \mathrm{V}(E)\}$. Reciprocally, any finite order separation $\{A,B\}$ gives rise to the \textit{$A$- and $B$-flaps} $E_A \defeq \mathrm{E}(A) \cup \mathrm{E}(A,A\cap B)$ and $E_B \defeq \mathrm{E}(B) \cup \mathrm{E}(B,A\cap B) $, where $\mathrm{E}(X)$ are edges with both tips in the vertex-set $X$ and $\mathrm{E}(X,Y)$ is the set of edges with one tip in $X$ and the other in $Y$. Both $E_A$ and $E_B$ will be in $\presep{G}$. They are not exactly complements, but the symmetric difference between $E_B$ and $E(G) \setminus E_A$ is finite - so when we consider the quotient $\separationalgebra{G}$, they will be complement to each other, as in, $E_A+E_B = 1_{\mathfrak{T}(G)}\in \mathfrak{T}(G)$ and $E_A \cdot E_B = 0_\mathfrak{\mathfrak{T}(G)}$. 

This describes the elements of $\separationalgebra{G}$ as `oriented' separations, with the operation of complement resulting in the opposite orientation of the separation. It also shows the operations happen `modulo' finite sets.

\begin{example}[Algebra of separations of a ray and of an infinite star]
\begin{itemize}
\item[(i)] Take a graph $R$ composed of a single ray. If a set $E \subset E(R)$ contains infinitely many edges of the ray, for it to define a finite separation it must contain all but a finite number of edges of the ray. With that we can conclude $\presep{R}$ is isomorphic to the Boolean algebra $\{A \subset \mathbb{N} \st A \text{ is finite or } \mathbb{N} \setminus A \text{ is finite }\}$, and therefore the quotient $\separationalgebra{R}$ must be isomorphic to the trivial Boolean algebra $\{0,1\}$.
\item[(ii)] Take a graph $S$ which is an infinite star with center $c$ and tips $n \in \mathbb{N}$. Any set of edges $E \subset E(S)$ defines a finite separation, so $\presep{S} \cong 2^{\mathbb{N}}$. So the quotient $\separationalgebra{S}$ is isomorphic to the Boolean algebra of subsets of $\mathbb{N}$ under `almost equality' - where two subsets are considered equal if their symmetric difference is finite.
\end{itemize}
\end{example}


\paragraph{The ultrafilter space of the separation algebra}
Following the usual construction of an ultrafilter space of a Boolean algebra, we define

\begin{definition}
    Denote as $\widetilde{\mathcal{T}}(G)\doteq \mathrm{Ult}(\presep{G})$ and $\mathcal{T}(G) \doteq \mathrm{Ult}(\mathfrak{T}(G))$. This is naturally given a topology by the $\mathrm{Ult}$ functor. 
\end{definition}

Every finite set of edges lies in $\separationalgebra{G}$, so in particular singletons $\{e\} \in \separationalgebra{G}$. Therefore, we have a principal ultrafilter $u_e \in \widetilde{\mathcal{T}}(G)$ associated with each atom $\{e\}$. As was shown in \myref{quotient}, we can consider $\mathcal{T}(G) \subset \widetilde{\mathcal{T}}(G) $ - and like we exemplified in \myref{quotientisremoval}, the principal ultrafilters $u_e$ are exactly the ones removed in this inclusion.

Now, the tangles of a graph are given a topology with their inverse limit characterization. We shall make use of the explicit association \[\tau \in \Theta(G) \leftrightarrow x_\tau \in \varprojlim \Phi_G\] given in \myref{inverselimitcharact}. Given a tangle $\tau$, and a finite set of vertices $F$ and component subset $A \subset \mathrm{S}_{G,F}$, notice $\{(\bigcup A) \cup F, (\bigcup (\mathrm{S}_{G,F}\backslash A)\cup F)\}$ is a finite separation, with separator $F$, so \[x_\tau(F) \doteq \{A \subset \mathrm{S}_{G,F} \st ( (\bigcup (\mathrm{S}_{G,F}\backslash A)\cup F),(\bigcup A) \cup F) \in \tau\}\]

\begin{figure}[ht]
\centering
\begin{tikzpicture}[scale = 0.5]
    \begin{scope}{scale = 1.5}
    \def\a{2}
    
    \draw (0,0) circle ({1.3*\a});
    \node[shift = {({-1.1*\a},0)}] at (-\a,0) {$\bigcup(\mathrm{S}_{G,F}\backslash A)\cup F)$};
    \node[above,magenta] at (0,{1.2*\a}) {$\mathrm{S}_{G,F} \backslash A$};
    
    \draw (2*\a,0) circle ({1.3*\a});
    \node[shift = {({3.1*\a},0)}] at (-\a,0) {$(\bigcup A) \cup F)$};
    \node[above,teal] at (2*\a,{1.2*\a}) {$A$};

    \draw[->] (0,\a) -- ({2*\a},\a) node[midway,above] {$\tau$};

    \draw[magenta] (-0.2*\a,0.7*\a) circle ({0.2*\a});
    \draw[magenta] (-0.2*\a,0.1*\a) circle ({0.2*\a});
    \draw[magenta] (-0.2*\a,-0.5*\a) circle ({0.2*\a});
    \node at (-0.2*\a,-0.9*\a) {$\vdots$};

    \draw[teal,fill] (2.2*\a,0.7*\a) circle ({0.2*\a});
    \draw[teal,fill] (2.2*\a,0.1*\a) circle ({0.2*\a});
    \draw[teal,fill] (2.2*\a,-0.5*\a) circle ({0.2*\a});
    \node at (2.2*\a,-0.9*\a) {$\vdots$};

    \node at (\a,0) {$F$};
    \end{scope}
\end{tikzpicture}
\caption{Given a finite set $F$, a tangle $\tau$ give rise to a \textcolor{teal}{filter} in $\mathrm{S}_{G,F}$ by choosing the set of the filled connected components}
\end{figure}
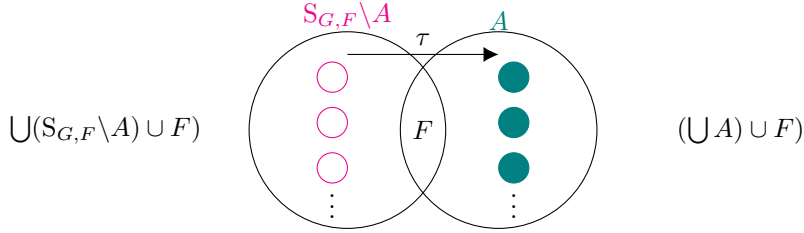

Furthermore, a sub-basis of the topology of $\varprojlim \Phi_G$ around this point is the collection given by two parameters $F \in [\mathrm{V}(G)]^{<\aleph_0}$ and $A \in \mathrm{S}_{G,F}$ \[ x_\tau \in \pi_F^{-1}[A] = \left\{x \in \varprojlim \Phi_G \st x(F) \in [A] \right\} \]

With that in mind, let's show that the ultrafilter space of the separation algebra is exactly the tangle space.

\begin{theorem}[Tangle algebra]
\label{tanglealgebra}
   Let $G$ be a connected graph. There is a homeomorphism $\tau_\bullet: \mathcal{T}(G) = \mathrm{Ult}(\mathfrak{T}(G))\to \Theta(G)$.
\end{theorem}

\begin{proof}
    We construct a tangle from a non-principal ultrafilter of finite separations $u$. Note that no finite set of edges can be an element of this filter, as was argued in the paragraph preceding this proof. Given $\{A,B\}$ a finite separation, we know its flaps $E_A,E_B \in \mathfrak{T}(G)$ sum to $1_{\separationalgebra{G}}$ - so $u$ picks one and only one between these two. This permits us to define an orientation \[\tau_u \defeq \{(A,B) \mid E_B \in u\}.\] It cannot have contradictions: if the orientations $(A_1,B_1), \dots, (A_n,B_n) \in \tau_u$ were to form a contradiction, then $E_{B_1} \cdots E_{B_n} = 0_{\mathfrak{T(G)}}$ which would contradict $u$ being a filter.\\
    
    Given a $\aleph_0$-tangle $\tau$, we define $u_\tau \defeq \{A \in \mathfrak{T}(G) \st (V(E(G) \setminus A),V(A)) \in \tau \}$ which we claim is an ultrafilter in $\mathcal{T}(G)$. Notice that this is well defined by \ref{almostequalseparations}.  Since $(\mathrm{V}(G),\emptyset) \notin \tau$, we have $0_{\mathfrak{T}(G)} \notin u_\tau$. Take two $A_1,A_2 \in u_\tau$, then $(V(E(G) \setminus A_i),V(A_i)) \in \tau$ and by \myref{intersectionclosure} we have \[(V(E(G) \setminus A_1) \cup V(E(G) \setminus A_2) ,V(A_1) \cap V(A_2)) \in \tau\] 
    So, since $V(A_1) \cap V(A_2)\sim V(A_1 \cap A_2)$, we can see that $A_1 \cap A_2 \in u_\tau$. The fact that $\tau$ is an orientation guarantees that for any $A \in \separationalgebra{G}$ we have that either $A$ or $E(G) \setminus A$ is an element of $u_\tau$, which finalizes the proof that $u_\tau$ is an ultrafilter. It is routine to show that these constructions are inverse to each other: $u_{\tau_{u}} = u$ and $\tau_{u_{\tau}} = \tau$.\\

    The basic open sets of the inverse limit $\Theta(G)$ are $O(F,C) = \{\tau \st C \in \tau_F$, where $F$ is a finite set of vertices, $C$ is a subset of connected components of $G \setminus F$, $\tau_F$ is the ultrafilter corresponding to $F$ of the tangle $\tau$. The basic open sets of $\mathcal{T}(G)$ are $O(E) = \{u \st E \in u\}$ where $E \in \separationalgebra{G}$. It is routine to see that the bijection we constructed takes the basic open set $O(F,C)$ to the basic open set $O(E_C)$, where $E_C$ is the set of edges of the connected components of $C$. Similarly, the bijection we constructed takes the basic open set $O(E)$ to the basic open set $O(F_E, C_E)$ where $F_E = V(E)\cap V(E(G) \setminus E)$ and $C_E = \{C \in S_{G,F} \st C \subset V(E)\}$. This means that the bijection is indeed a homeomorphism.
    


\end{proof}

\begin{example}[Tangle space of a ray and of an infinite star]
\begin{itemize}
\item[(i)] Take a graph $R$ composed of a single ray. As we already argued $\presep{R}$ is isomorphic to the Boolean algebra $\{A \subset \mathbb{N} \st A \text{ is finite or } \mathbb{N} \setminus A \text{ is finite }\}$, and therefore the quotient $\separationalgebra{R}$ is isomorphic to the trivial Boolean algebra $\{0,1\}$. So $\mathrm{Ult}(\presep{R})$ is the 1-point compactification of $\mathbb{N}$, and $\mathrm{Ult}(\separationalgebra{R}) \cong \Theta(R)$ is a space of a single point.
\item[(ii)] Take a graph $S$ which is an infinite star with center $c$ and tips $n \in \mathbb{N}$. As we already argued, $\presep{S}$ is isomorphic to $2^{\mathbb{N}}$. So $\mathrm{Ult}(\presep{S})$ is the space of ultrafilters of natural numbers, that is, $\beta \mathbb{N}$. When we take the quotient $\separationalgebra{S}$, its ultrafilter space will contain all the ultrafilters except the principal ones, so  $\mathrm{Ult}(\separationalgebra{S}) \cong \Theta(S) $ is homeomorphic to $\beta \mathbb{N} \setminus \mathbb{N}$.
\end{itemize}
\end{example}

\section{Edge-tangles, edge-directions and the cut algebra}

\subsection{Edge-tangles and edge-directions}
Consider the tangle characterization given in  \myref{inverselimitcharact}, which is a way to define a topological structure on the set of tangles. If we change \emph{finite vertex-sets} to \emph{finite edge-sets}, we can define, for a connected graphs $G$:

\begin{definition}[Edge-tangles]
    Define the inverse system $\Phi_G^E$ in the same way as $\Phi_G$, but using $[\mathrm{E}(G)]^{< \aleph_0}$ instead of $[\mathrm{V}(G)]^{< \aleph_0}$. Call the inverse limit of the system the \textbf{edge-tangle space} $\Theta_E(G)$ and its elements edge-tangles.
\end{definition}

Now, for any finite edge-set $F \subset \mathrm{E}(G)$ and $G$ connected, we have that $S_{G,F}$ is always finite, so all the filters over $2^{\mathrm{\mathrm{S}_{G,F}}}$ are principal. This means that an edge-tangle is the same thing as an edge-direction, as defined in \cite{meninos2025}:

\begin{definition}[Edge-directions]
    Let $G$ be a graph. Define $\mathrm{S}^E_{G,F}$ to be the discrete space of the connected components of $G\backslash F$ for $F \in [\mathrm{E}(G)]^{<\aleph_0}$; take as transition maps $\varepsilon^G_{F_1,F_2}: \mathrm{S}^E_{G,F_2} \rightarrow \mathrm{S}^E_{G,F_1} $, for $F_1 \subset F_2$, the map that takes a connected component from $\mathrm{S}^E_{G,F_2}$ to the connected component from $\mathrm{S}^E_{G,F_1}$ in which it is contained in. The inverse limit of the system $(\mathrm{S}^E_{G,F}, \varepsilon^G_{F_1,F_2})$ will be the topological space of \emph{edge-directions} of $G$, denoted by $\mathcal{D}_E(G)$. We call it the \emph{edge-direction space}. It is routine to see that $\set{\mathcal{D}_E(F,\rho):F\in [\mathrm{E}(G)]^{\aleph_0}, \, \rho\in \mathcal{D}_E(G)}$ is a basis for $\mathcal{D}_E(G)$, with
\[\mathcal{D}_E(F,\rho) \doteq \set{\rho'\in \mathcal{D}_E(G)\st \rho'(F) = \rho(F)}.\]
\end{definition}

\begin{proposition}[Edge-tangles are edge-directions]
    There is a homeomorphism $\Theta_E(G) \cong \mathcal{D}_E(G)$.
\end{proposition}

Similarly to the vertex case, the edge-tangle/edge-direction space provides a compact space in which the edge-end space lies in.

\begin{proposition}[\cite{meninos2025}]
    $\mathcal{D}_E(G)$ is a compact, zero dimensional Hausdorff space, i.e., a Stone space. There is a topological inclusion $\Omega_E(G) \rightarrow \mathcal{D}_E(G)$. This is a homeomorphism when $G$ is locally finite.
\end{proposition}

\paragraph{Line graphs and edge-directions} Given any graph $G$, we can construct the \emph{line graph of $G$}, denoted by $L(G)$, to be the graph whose vertices are in correspondence with $G$-edges (that is, $\mathrm{V}(L(G))\doteq \{v_e\}_{e \in \mathrm{E}(G)}$) such that $\{v_{e_1},v_{e_2}\}$ is an edge of $G'$ exactly when $e_1$ shares a vertex with $e_2$. Using the notation $X \approx Y$ for the homeomorphism relation between spaces, we prove in \cite{meninos2025} the following.

\begin{theorem}[\cite{meninos2025}]
\label{linetheorem}
    For any graph $G$, $\mathcal{D}_E(G) \approx \Omega(L(G))$. 
\end{theorem}

Although useful in the topological context, edge-directions are difficult to manipulate in a `combinatorial way' because they lack the `ray representative' that is frequently used in the study of the usual vertex version of directions. The line graph end characterization should be useful in these situations. We present one way in which it is useful, where we characterize edge-directions as either arising from rays or from non-dominating vertices of infinite degree - a sort of edge analogue of \cite{DIESTEL2003197}'s bijection result.


\begin{corollary}[{of \myref{linetheorem}}, \cite{meninos2025}]
\label{directionrepresentation}
    Let $\rho \in \mathcal{D}_E(G)$ be any direction, then either one of the following is true, but not both:
    \begin{itemize}
        \item[(i)] There is a ray $r$ such that, for any finite edge-set $F$, $\rho(F)$ is the component of $G \backslash F$ that has an $r$-tail.
        \item[(ii)] There is a vertex $v$ of infinite degree that does not dominate any ray and such that, for any finite edge-set $F$, $\rho(F)$ is the component $G \backslash F$ containing $v$.
    \end{itemize}
\end{corollary}
    \subsection{Cut algebras}

In what follows, we give an edge analogue of \myref{tanglealgebra}. Letting $A \subset V(G)$ be a set of vertices, we denote \[E(A,G \backslash A) \doteq \{e \in \mathrm{E}(G) \st e \text{ has one vertex in } A \text{ and the other in } G \backslash A\}\] Note that this set of edges, when removed from the graph, disconnects it. This is usually called a cut.

For any (connected) graph $G$, let $I_G \subset 2^{\mathrm{V}(G)}$ be the ideal of all the finite sets of vertices of finite degree. 

\begin{definition}[Pre-cut algebra and cut algebra]
    For a (connected) graph $G$, the \textbf{finite pre-cut algebra of $G$} is \[\precut{G} \doteq \{A \subset \mathrm{V}(G) \st E(A,G \backslash A) \text{ is finite }\}\] considered with the usual operations of union, intersection and complement (in finite graph theory, this is known as \textbf{the bond space} of $G$, a vector space over $\mathbb{F}_2$). The finite cut algebra of $G$ is $\cutalgebra{G} \doteq \precut{G}/I_G$.
\end{definition}

\begin{rmk*} 
Although we took a quotient in the definition, we shall denote elements $A \in \cutalgebra{G}$ by any of its representatives. 
\end{rmk*}

Like in \myref{quotientisremoval}, we will regard  $\mathrm{Ult}(\cutalgebra{G})$ as the subset of $\mathrm{Ult}(\precut{G}) $ that contains the ultrafilters that are not principal ultrafilters of vertices of finite degree.

\begin{theorem}[Edge-tangle algebra]
\label{directionhomeo}
For a connected graph $G$, there is a homeomorphism $\mathrm{Ult}(\cutalgebra{G}) \approx \mathcal{D}_E(G)$.
\end{theorem}

\begin{proof}

Let us define a map $u_{\bullet}:\mathcal{D}_E(G) \to \mathrm{Ult}(\cutalgebra{G})$. Using \myref{directionrepresentation}, we know a $\rho \in \mathcal{D}_E(G)$ either comes from a ray $r$ or a non-dominating vertex $v$ of infinite degree. Given a ray $r$, define $u_r$ as $\{B \in \cutalgebra{G}  \st B \text{ contains a tail of } r $\}. This is well defined due to representatives of the same quotient class only differing by a finite number of elements - so one contains a tail if and only if the other one does. The empty set cannot be in $u_r$, an intersection of two sets in $u_r$ must contain the smallest between the two tails of each set, and between $B$ and $G \setminus B$, a ray must have a tail entirely in one and only one of them, otherwise $E(B, G\setminus B)$ would be infinite - so $u_r$ is an ultrafilter. If we had picked a different ray $r'$ that gives rise to the same direction $\rho$, then the rays would be infinitely edge-connected (as is argued in \cite{meninos2025}), so they give rise to the same ultrafilter. Now, given a non dominating vertex $v$ of infinite degree, we define $u_v$ as the principal ultrafilter associated with $v$. Similarly to the case of rays, any two vertices that define the same direction cannot be separated by finitely many edges, so the ultrafilters they define will have to be the same. Neither $u_r$ or $u_v$ can end up being a principal ultrafilter of a vertex of finite degree, since they only pick infinite sets. With this, we have constructed a well defined map $u_{\bullet}:\mathcal{D}_E(G) \to \mathrm{Ult}(\cutalgebra{G})$. \\

Let us define a map $\rho_{\bullet}:\mathrm{Ult}(\cutalgebra{G}) \to \mathcal{D}_E(G)$. Take $u \in \mathrm{Ult}(\cutalgebra{G})$ and $F$ a finite set of edges. The removal of $F$ separates $G$ into a finite set of components $\mathrm{S}^\mathrm{E}_{G,F}$. Each $C \in \mathrm{S}^\mathrm{E}_{G,F} $ is only finitely connected to its complement, so $C \in \cutalgebra{G}$. Since the union of all these $C$ result in the entirety of $G$, $u$ must pick at least one of them, and cannot pick more than one since they are all disjoint. If $C$ is that chosen component, define $\rho_u(F) = C$. This will be a direction: if we were to remove more edges $F \subset F'$, the new components will be either entirely contained in $C$ or disjoint from it. Since the ultrafilter cannot pick one disjoint from $C$, then $\rho_u(F') \subset \rho_u(F)$. Also, since $u$ only picks infinite sets, $\rho_u$ only picks infinite components.\\

It is a routine calculation to convince oneself that these constructions $u_\bullet$ and $\rho_\bullet$ are inverse to each other.\\

Take $u \in \mathrm{Ult}(\cutalgebra{G})$ and a basic neighborhood $[B]$ containing $u$. Letting $F = E(B, G \setminus B)$, we already argued that $u$ picks only one of the connected components of $G \setminus F$, let's call it $C$. Then $u \in [C] \subset [B]$. This shows that the open sets $[C]$ arising from connected components $C$ of $G \setminus F$ for all the finite sets of edges $F$ form a basis for the space $\mathrm{Ult}(\cutalgebra{G})$. Notice now that the map $\rho_\bullet$ will take $[C]$ to the basic open set $\mathcal{D}_E(\rho_u,F)$, and the map $u_\bullet$ takes $\mathcal{D}_E(\rho_u,F)$ back towards $[C]$.
\[
u \in [C] 
\xrightarrow[\rho_{\bullet}]{\xleftarrow{u_{\bullet}}}
\rho_\mathcal{U} \in \mathcal{D}_E(\rho_\mathcal{U}, F),
\]

This shows that these bijections take basis sets to basis sets, thus concluding the proof that they are homeomorphisms.
\end{proof}

\begin{corollary}
\label{object-commute}
    For a connected graph $G$, the spaces $\mathcal{D}_E(G) \approx \Omega(L(G)) \approx \mathrm{Ult}(\mathfrak{D}(G))$ are all homeomorphic. In a categorical phrasing, this means that the diagram
    \[
    \begin{CD}
        \catname{Graph}  		@>L>> 	\catname{Graph} 		\\
        @V \cutalgebrasymb VV 			@VV \Omega V  	\\
        \catname{Bool}			@>\mathrm{Ult} >> 	\catname{Top}
    \end{CD}
    \]
    commutes, considering these only as functions from the objects of each of these categories (connected graphs for $\catname{Graph}$, Boolean algebras for $\catname{Bool}$ and topological spaces for $\catname{Top}$).
\end{corollary}

    \section{The edge-direction functor }
\label{edgedirecfuncsection}

Functorial associations between objects of a combinatorial, algebraic or topological nature are ubiquitous in literature: the Stone space of a Boolean algebra, the fundamental group of a topological space, the geometric realization of an abstract simplicial complex... These associations offer tools and properties of these objects. From \myref{object-commute}, we have the commutativity of the diagram

    \[
    \begin{CD}
        \catname{Graph}  		@>L>> 	\catname{Graph} 		\\
        @V \mathbb{\cutalgebrasymb} VV 			@VV \Omega V  	\\
        \catname{Boole}			@>\mathrm{Ult} >> 	\catname{Top}
    \end{CD}
    \]
    
but only for the objects of these categories. In what follows, we ask if we can make the commutativity of the above diagram a functorial one. The functorial nature of $\mathrm{Ult}$ is well known - it is the one used in the categorical equivalence between Boolean algebras and Stone spaces. The other maps, though, are not widely known to be functorial. What we will do is choose a specific class of morphisms in $\catname{Graph}$ where $\mathbb{\cutalgebrasymb}, L$ and $\Omega$ will have naturally defined images for these morphisms. Thus, when restricted to this class, the above diagram will be a true categorical commuting diagram - and we will have built, through each of the two possible pathways, a functor $\mathcal{D}_E: \catname{Graph} \rightarrow \catname{Top}$ mapping a graph into its edge-direction space.

\begin{rmk*}
    Some of the results do not require the graph to be connected and some do - for simplicity, we will define our categories to always contain only connected graphs.
\end{rmk*}

\subsection{Making \texorpdfstring{$\Omega$}{Lg} into a functor}

In what follows, we show how to functorialize $\Omega$. This was explored by one of the authors and colleagues in \cite{funcPaper}. We will present here only the necessary for what we will use, and the interested reader can find a more in-depth discussion in \cite{funcPaper}. \\

$\catname{Graph}$ is usually regarded as a category with morphisms $\phi:G_1 \rightarrow G_2$ being given by graph homomorphisms or 1-complex morphisms. The first class is defined by demanding $\phi$ takes adjacent vertices (vertices joined together by an edge) to adjacent vertices, while the second one relaxes this demand by allowing adjacent vertices to be sent either to adjacent vertices or to a single vertex (one way to think of this is to extend the notion of `adjacency' by allowing a vertex to be called adjacent to itself). Not all of these morphisms are well-behaved enough to induce a well-defined continuous map between end spaces. In \cite{funcPaper}, it is described two classes of morphisms where $\Omega$ will be a functor. We will consider here only one of them: the strongly proper morphisms.

\begin{definition}[Dispersed sets and strongly proper morphisms, \cite{funcPaper}]
    A set of vertices $U \subset V(G)$ is called \textbf{dispersed} when, for any ray $r$ of $G$, there exists a finite set of vertices $F \subset V(G)$ such that the connected component $C$ of $G \setminus F$ which contains a tail of $r$ is disjoint from $U$. In other words, $U$ is dispersed when it can be separated from any ray using finitely many vertices.\\
    Let $\phi:G_1 \to G_2$ be a 1-complex morphism. We say $\phi$ is \textbf{strongly proper} when $\phi^{-1}(v)$ is dispersed for every vertex $v \in V(G_2)$.
\end{definition}

\begin{example}[Strongly proper morphisms, \cite{funcPaper}]
\label{strproperexamples}
    The following morphisms are all strongly proper.
    \begin{itemize}
        \item[(i)] The identity map of any graph.
        \item[(ii)] Injective graph homomorphisms (which can be seen as subgraph inclusions).
        \item[(iii)] An 1-complex morphism $\phi$ such that $\phi^{-1}(v)$ is always finite, for any $v \in V(G_2)$
        \item[(iv)] A graph homomorphism $\phi$ between pruned trees of height $\omega$ such that $\phi^{-1}(v)$ is rayless for any $v \in V(G_2)$.
    \end{itemize}
\end{example}  

Strongly proper morphisms have nice `end preserving properties' in such a way that it allows us to induce from them a continuous map between the ends. Here's one somewhat immediate such property. A graph morphism $\phi$ takes a ray $r \subset G_1$ to a path $\phi(r)$ that is not necessarily a ray, since $\phi$ can be non-injective. However, if $\phi$ is strongly proper, then there is a ray $r' \subset \phi(r)$. This ray $r'$ induces an unique end $\varepsilon$, as the following proposition describes.

\begin{proposition}[\cite{funcPaper}]
    Let $\phi:G_1 \to G_2$ be a strongly proper morphism. For every $G_1$-ray $r$, there is a ray $s_r \subset \phi(r)$ with the following properties.
    \begin{itemize}
        \item[(i)] \textbf{In $G_2$:} Any ray $s' \subset G_2$ that adheres to $\phi(r)$ is equivalent to $s_r$.
        \item[(ii)] \textbf{In $G_1$:} Any ray $r' \subset r$ such that $r' \sim r$ defines a ray $s'_{r'}\subset \phi(r')$ such that $s_r \sim s'_{r'}$.
    \end{itemize}
    Therefore, there is a well defined function from $\Omega(G_1)$ taking an end induced by a ray $r$ to the end in $\Omega(G_2)$ induced by the ray $s_r$, which is the only end that adheres to $\phi(r)$.
    \label{properraymap}
\end{proposition}

\begin{proof}
    The image $\phi(r)$ can't be finite, otherwise $\phi$ would not be strongly proper. By the star-comb lemma (\myref{star-comb}) applied to the \emph{connected} graph $G_2[\phi(r)]$, there is either a comb or an infinite star subdivision in $G_2[\phi(r)]$. If there wasn't a comb then we would definitely have a star, let's say with center $c$. There can't be infinitely many vertices $v \in r$ with $\phi(v) = c$ because that would mean $\phi$ is not strongly proper. On the other hand, $\phi^{-1}(c)$ can't be finite either, otherwise a tail of $r$ would have to be mapped into a single leg of the star, which is a finite set, again contradicting $\phi$ being strongly proper. So we conclude there must be a comb inside $\phi(r)$, let's say with spine $s_r \subset \phi(r)$.\\

    $(i)$: Consider any other ray $s'$ that adheres to $\phi(r)$. This means there is a $G_2$-comb $\{s',\{p_n\}\}$ with $s'$ as a spine and infinitely many teeth $U \subset \phi(r)$. Suppose $s' \not\sim s_r$. In that case, there must be a finite vertex-set $F \subset V(G_2)$ that separates them. This means that in $G_2 \backslash F$, $s'$ has a tail in a connected component $C_1$ and $s_r$ in another $C_2$. Since $\phi^{-1}(F)$ is dispersed, then $r$ has a tail $r'\subset r$ such that $\phi(r')$ is entirely contained in a connected component of $G_2 \setminus F$, which must necessarily then be $C_2$. But this contradicts the existence of infinitely many pairwise disjoint paths $p_n$ from $s'$ to $\phi(r)$.\\

  $(ii)$: Let $r' \sim r$. We prove the ray $s_r$ adheres to $\phi(r')$, from which, by using $(i)$, we can conclude $s_{r} \sim s'_{r'}$. Suppose not, in other words, suppose there is a finite $F \subset V(G_2)$ that separates $\phi(r')$ from $s_r$. Since $r' \sim r$, there must be infinitely many pairwise disjoint paths $\{p_n\}$ where each $p_n$ starts at a vertex in $r$ and ends in a vertex in $r'$. This means $\phi(p_n)$ is a \emph{walk} (a path that can repeat vertices) from $\phi(r')$ to $\phi(r)$. Since $\phi^{-1}(F)$ must be dispersed, there is a tail of $r$ and a tail of $r'$ such that those tails are disjoint from $\phi^{-1}(F)$ and so are the paths $p_n$. But then this means that in $G_2 \setminus F$, a tail of the spine $s_r \subset \phi(r)$ is in the same connected component as $\phi(r')$, contradicting the choice of $F$ as the set of vertices that separates $\phi(r')$ from $s_r$.
\end{proof}

\begin{definition}[The end-map]
    Let $\phi:G_1 \to G_2$ be a strongly proper map. The previous proposition gives rise to a well defined function $\Omega \phi: \Omega(G_1) \to \Omega(G_2)$, where for each ray $r$ we define $\Omega \phi([r]) = [s]$ where $s$ is any ray $s \subset \phi(r)$.
\end{definition} 

We would like to form a category with strongly proper morphisms as the morphisms. So, we need to make sure that composition of strongly proper morphisms is still strongly proper. This follows from

\begin{proposition}[\cite{funcPaper}]
\label{thinfiber}
    Let $\phi:G_1 \to G_2$ be strongly proper and $D \subset \mathrm{V}(G_2)$ dispersed. Then $\phi^{-1}(D)$ is dispersed.
\end{proposition}
\begin{proof}
    Suppose there was a ray $r \subset G_1$ that adheres to $\phi^{-1}(D)$. This means we have infinitely many pairwise disjoint paths $\{p_n\}$ starting from $r$ and ending at $\phi^{-1}(D)$. Applying $\phi$ to $r$, we get a connected subset $\phi(r)$ that is not necessarily a ray, but contains a ray $r' \subset \phi(r)$, as we argued in the last proposition. By applying $\phi$ to the paths $\{p_n\}$, we get connected subsets $\phi(p_n)$ which must contain a path $q_n$ from the start in $\phi(p_1)$ to the end in $\phi(p_n)$. Therefore $q_n$ connects $\phi(r)$ to $D$. Since $D$ is dispersed, $r'$ can't adhere to $D$, so there is a finite $F \subset V(G_2)$ that separates $r'$ from $D$. The paths $q_n$ must go through $F$ then, which implies $\phi^{-1}(F)$ contains infinitely many elements from the paths $p_n$. But this would imply $r$ adheres to $\phi^{-1}(F)$, contradicting $\phi$ being strongly proper.
\end{proof}

If $\phi:G_1 \to G_2$ and $\psi:G_2 \to G_3$ are strongly proper, then if we take a a finite $F \subset V(G_3)$, we will have $(\psi \circ \phi)^{-1}(F) = \phi^{-1}(\psi^{-1}(F))$ dispersed due to $\psi^{-1}(F)$ being dispersed - showing that $\psi \circ \phi$ is also strongly proper.

\begin{proposition}[\cite{funcPaper}]
    Let $\phi:G_1 \to G_2$ be a strongly proper morphism. The end-map $\Omega(\phi): \Omega(G_1) \rightarrow \Omega(G_2)$ is continuous.
\end{proposition}

\begin{proof}
A basic neighborhood of $\Omega(G_2)$ has the form $\Omega(\epsilon',F) \subset \Omega(G_2)$ for a certain $\epsilon' \in \Omega(G_2)$ and a finite $F \subset V(G_2)$. Since $\phi^{-1}(F)$ is dispersed, we can use the technique of enveloping, which can be found in \cite{duality}, to find a set $D \subset V(G_1)$ such that\footnote{$D$ will in fact be a rayless normal sub-tree of $G$, but we will not need to use that.}
\begin{itemize}
    \item $\phi^{-1}(F) \subset D$
    \item $D$ has finite adherence, which means that for each $C$ connected component of $G_1 \setminus D$ there exists a finite set $F_{D,C}$ such that any path starting at $C$ and ending at $G_1 \setminus C$ must go through a vertex in $F_{D,C}$
    \item $D$ is dispersed
\end{itemize}

For a given $\epsilon \in \Omega(\phi)^{-1}(\Omega(\epsilon',F))$, say that $C$ is the connected component containing a tail of any representative of $\epsilon$ in $G_1 \setminus D$ (this must exist since $D$ is dispersed). We claim that $ \Omega(\epsilon,F_{D,C}) \subset \Omega(\phi)^{-1}(\Omega(\epsilon',F))$, which is enough to conclude $\Omega(\phi)$ is continuous. Take a $\eta \in \Omega(\epsilon,F_{D,C})$. Notice that $C$ is also a connected component of $G_1 \setminus F_{D,C}$ and it contains a tail of any representative of $\epsilon$ and therefore a tail of any representative of $\eta$. Also, $C \subset G_1 \setminus D \subset G_1 \setminus \phi^{-1}(F)$. Take rays $r_\epsilon$ and $r_\eta$ that induce $\epsilon$ and $\eta$ respectively. The images $\phi(r_\epsilon)$ and $\phi(r_\eta)$ are connected subsets of $G_2$ that contain rays $r_{\epsilon'}$ and $r_{\eta'}$  that induce, respectively, $\Omega(\phi)(\epsilon) = \epsilon'$ and $\Omega(\phi)(\eta) \defeq \eta'$. Since there is a path between $r_\epsilon$ and $r_\eta$ in $C$, this path is entirely outside of $\phi^{-1}(F)$, and by mapping that path with $\phi$ to $G_2$, we get a path between $r_{\epsilon'}$ and $r_{\eta'}$ that does not go through $F$. Therefore $\eta' \in \Omega(\epsilon',F)$, concluding our claim.

\end{proof}

It is quite clear that $\Omega \id_G = \id_{\Omega(G)}$, since $r \subset r = \id_{G}(r)$. The functoriality is also readily available for us now.

\begin{proposition}
    Let $\phi:G_1 \to G_2$ and $\psi:G_2 \to G_3$ be strongly proper maps, then $\Omega(\psi \circ \phi) = \Omega \psi \circ \Omega \psi$.
\end{proposition}

\begin{proof}
    Let $r$ be any $G_1$-ray. There exists a $G_2$-ray $s \subset \phi(r)$, and there exists a $G_3$-ray $t$ such that \[t \subset \psi(s) \subset \psi(\phi(r)) = (\psi\circ \phi)(r).\] Notice $[s] = \Omega \phi ([r])$ and finally \[ [t] = \Omega \psi ([s]) = \Omega \psi (\Omega \phi ([r])) = (\Omega \psi \circ \Omega \phi) ([r]).\]
\end{proof}

With everything we showed, we have

\begin{theorem}[The strongly proper category, \cite{funcPaper}]
\label{end-spacefunc}
    There is a category $\catname{Graph}_{SP}$  whose objects are connected graphs and whose morphisms are the strongly proper morphisms. There is a functor $\Omega:\catname{Graph}_{SP} \to \catname{Top}$ that takes a graph $G$ to its end space $\Omega(G)$ and takes a strongly proper morphism $\phi:G_1 \rightarrow G_2$ to a continuous map $\Omega(\phi): \Omega(G_1) \rightarrow \Omega(G_2)$. We can describe the map $\Omega(\phi)$ explicitly: for any end $\varepsilon \in \Omega(G_1)$ and a ray $r$ that induces $\varepsilon$, the image $\Omega(\phi)(\varepsilon)$ is the only end that cannot be separated from $\phi(r)$ by removing a finite number of vertices; also, there exists a ray $r' \subset \phi(r)$ that induces $\Omega(\phi)(\varepsilon)$. 
\end{theorem}

\subsection{Making \texorpdfstring{$L$}{Lg} into a functor }
Take a morphism $\phi:G_1 \to G_2$. For there to be a map $L(\phi):L(G_1) \to L(G_2)$ defined in the natural way $L(\phi)(\{u,v\}) \doteq \{\phi(u),\phi(v)\}$ we have to demand that $\phi$ maps vertices joined together by an edge into vertices joined together by an edge - as in, $\phi$ be a graph homomorphism. If two edges $a \doteq {\{x,y\}}$ and $b \doteq {\{y,z\}}$ are adjacent in $L(G_1)$, observe $L(\phi)(a)$ and $L(\phi)(b)$ will also be adjacent edges (they share the common vertex $\phi(y)$).

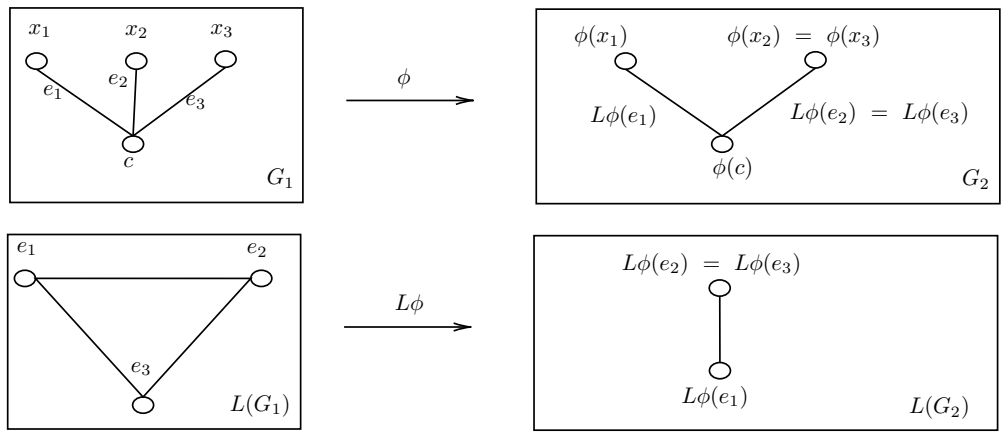
\begin{figure}[ht]
\centering
\scalebox{0.85}{
\tikzset{every picture/.style={line width=0.75pt}} 

\begin{tikzpicture}[x=0.75pt,y=0.75pt,yscale=-0.7,xscale=0.95]

\draw   (108,226.42) .. controls (107.99,222.65) and (110.9,219.57) .. (114.49,219.57) .. controls (118.09,219.56) and (121.01,222.61) .. (121.02,226.39) .. controls (121.03,230.17) and (118.12,233.24) .. (114.52,233.25) .. controls (110.93,233.26) and (108.01,230.2) .. (108,226.42) -- cycle ;
\draw   (48,156.42) .. controls (47.99,152.65) and (50.9,149.57) .. (54.49,149.57) .. controls (58.09,149.56) and (61.01,152.61) .. (61.02,156.39) .. controls (61.03,160.17) and (58.12,163.24) .. (54.52,163.25) .. controls (50.93,163.26) and (48.01,160.2) .. (48,156.42) -- cycle ;
\draw   (110,156.42) .. controls (109.99,152.65) and (112.9,149.57) .. (116.49,149.57) .. controls (120.09,149.56) and (123.01,152.61) .. (123.02,156.39) .. controls (123.03,160.17) and (120.12,163.24) .. (116.52,163.25) .. controls (112.93,163.26) and (110.01,160.2) .. (110,156.42) -- cycle ;
\draw   (165.47,154.78) .. controls (165.47,151) and (168.37,147.93) .. (171.97,147.92) .. controls (175.56,147.91) and (178.49,150.96) .. (178.49,154.74) .. controls (178.5,158.52) and (175.6,161.59) .. (172,161.6) .. controls (168.4,161.61) and (165.48,158.55) .. (165.47,154.78) -- cycle ;
\draw    (54.52,163.25) -- (114.49,219.57) ;
\draw    (116.52,163.25) -- (114.49,219.57) ;
\draw    (172,161.6) -- (114.49,219.57) ;
\draw   (474,226.42) .. controls (473.99,222.65) and (476.9,219.57) .. (480.49,219.57) .. controls (484.09,219.56) and (487.01,222.61) .. (487.02,226.39) .. controls (487.03,230.17) and (484.12,233.24) .. (480.52,233.25) .. controls (476.93,233.26) and (474.01,230.2) .. (474,226.42) -- cycle ;
\draw   (414,156.42) .. controls (413.99,152.65) and (416.9,149.57) .. (420.49,149.57) .. controls (424.09,149.56) and (427.01,152.61) .. (427.02,156.39) .. controls (427.03,160.17) and (424.12,163.24) .. (420.52,163.25) .. controls (416.93,163.26) and (414.01,160.2) .. (414,156.42) -- cycle ;
\draw   (532,155.42) .. controls (531.99,151.65) and (534.9,148.57) .. (538.49,148.57) .. controls (542.09,148.56) and (545.01,151.61) .. (545.02,155.39) .. controls (545.03,159.17) and (542.12,162.24) .. (538.52,162.25) .. controls (534.93,162.26) and (532.01,159.2) .. (532,155.42) -- cycle ;
\draw    (420.52,163.25) -- (480.49,219.57) ;
\draw    (538.52,162.25) -- (480.49,219.57) ;
\draw    (247,189.6) -- (301,189.6) -- (324,189.6) ;
\draw [shift={(326,189.6)}, rotate = 180] [color={rgb, 255:red, 0; green, 0; blue, 0 }  ][line width=0.75]    (10.93,-3.29) .. controls (6.95,-1.4) and (3.31,-0.3) .. (0,0) .. controls (3.31,0.3) and (6.95,1.4) .. (10.93,3.29)   ;
\draw   (38,111.6) -- (220,111.6) -- (220,275.6) -- (38,275.6) -- cycle ;
\draw   (365,112.6) -- (653,112.6) -- (653,276.6) -- (365,276.6) -- cycle ;
\draw   (114.23,446.48) .. controls (114.22,442.7) and (117.13,439.63) .. (120.72,439.62) .. controls (124.32,439.61) and (127.24,442.67) .. (127.25,446.44) .. controls (127.26,450.22) and (124.35,453.29) .. (120.76,453.3) .. controls (117.16,453.31) and (114.24,450.26) .. (114.23,446.48) -- cycle ;
\draw   (40.78,339.64) .. controls (40.77,335.86) and (43.68,332.79) .. (47.28,332.78) .. controls (50.87,332.77) and (53.79,335.82) .. (53.8,339.6) .. controls (53.81,343.38) and (50.9,346.45) .. (47.31,346.46) .. controls (43.71,346.47) and (40.79,343.41) .. (40.78,339.64) -- cycle ;
\draw   (187.68,339.63) .. controls (187.67,335.85) and (190.58,332.78) .. (194.18,332.77) .. controls (197.77,332.76) and (200.69,335.81) .. (200.7,339.59) .. controls (200.71,343.37) and (197.8,346.44) .. (194.21,346.45) .. controls (190.61,346.46) and (187.69,343.4) .. (187.68,339.63) -- cycle ;
\draw   (472.6,417.42) .. controls (472.59,413.65) and (475.5,410.57) .. (479.09,410.57) .. controls (482.69,410.56) and (485.61,413.61) .. (485.62,417.39) .. controls (485.63,421.17) and (482.72,424.24) .. (479.12,424.25) .. controls (475.53,424.26) and (472.61,421.2) .. (472.6,417.42) -- cycle ;
\draw   (472.47,347.78) .. controls (472.47,344) and (475.37,340.93) .. (478.97,340.92) .. controls (482.56,340.91) and (485.49,343.96) .. (485.49,347.74) .. controls (485.5,351.52) and (482.6,354.59) .. (479,354.6) .. controls (475.4,354.61) and (472.48,351.55) .. (472.47,347.78) -- cycle ;
\draw    (479,354.6) -- (479.09,410.57) ;
\draw    (245.6,380.6) -- (299.6,380.6) -- (322.6,380.6) ;
\draw [shift={(324.6,380.6)}, rotate = 180] [color={rgb, 255:red, 0; green, 0; blue, 0 }  ][line width=0.75]    (10.93,-3.29) .. controls (6.95,-1.4) and (3.31,-0.3) .. (0,0) .. controls (3.31,0.3) and (6.95,1.4) .. (10.93,3.29)   ;
\draw   (36.6,302.6) -- (218.6,302.6) -- (218.6,466.6) -- (36.6,466.6) -- cycle ;
\draw   (363.6,303.6) -- (651.6,303.6) -- (651.6,467.6) -- (363.6,467.6) -- cycle ;
\draw   (120.72,439.62) -- (53.8,339.6) -- (187.68,339.63) -- cycle ;

\draw (107,235) node [anchor=north west][inner sep=0.75pt]   [align=left] {$\displaystyle c$};
\draw (48,123) node [anchor=north west][inner sep=0.75pt]   [align=left] {$\displaystyle x_{1}$};
\draw (108,124) node [anchor=north west][inner sep=0.75pt]   [align=left] {$\displaystyle x_{2}$};
\draw (161,123) node [anchor=north west][inner sep=0.75pt]   [align=left] {$\displaystyle x_{3}$};
\draw (57,175) node [anchor=north west][inner sep=0.75pt]   [align=left] {$\displaystyle e_{1}$};
\draw (473,235) node [anchor=north west][inner sep=0.75pt]   [align=left] {$\displaystyle \phi ($$\displaystyle c)$};
\draw (387,124) node [anchor=north west][inner sep=0.75pt]   [align=left] {$\displaystyle \phi ( x_{1})$};
\draw (397,191) node [anchor=north west][inner sep=0.75pt]   [align=left] {$\displaystyle L\phi ( e_{1})$};
\draw (277,160) node [anchor=north west][inner sep=0.75pt]   [align=left] {$\displaystyle \phi $};
\draw (482,124) node [anchor=north west][inner sep=0.75pt]   [align=left] {$\displaystyle \phi ( x_{2}) \ =\ \phi ( x_{3})$};
\draw (97.52,165.25) node [anchor=north west][inner sep=0.75pt]   [align=left] {$\displaystyle e_{2}$};
\draw (521,189) node [anchor=north west][inner sep=0.75pt]   [align=left] {$\displaystyle L\phi ( e_{2}) \ =\ L\phi ( e_{3})$};
\draw (145.52,183.25) node [anchor=north west][inner sep=0.75pt]   [align=left] {$\displaystyle e_{3}$};
\draw (40.6,307) node [anchor=north west][inner sep=0.75pt]   [align=left] {$\displaystyle e_{1}$};
\draw (183.6,308) node [anchor=north west][inner sep=0.75pt]   [align=left] {$\displaystyle e_{2}$};
\draw (453.6,428) node [anchor=north west][inner sep=0.75pt]   [align=left] {$\displaystyle L\phi ( e_{1})$};
\draw (275.6,351) node [anchor=north west][inner sep=0.75pt]   [align=left] {$\displaystyle L\phi $};
\draw (416.6,316) node [anchor=north west][inner sep=0.75pt]   [align=left] {$\displaystyle L\phi ( e_{2}) \ =\ L\phi ( e_{3})$};
\draw (111.6,408) node [anchor=north west][inner sep=0.75pt]   [align=left] {$\displaystyle e_{3}$};
\draw (196,245) node [anchor=north west][inner sep=0.75pt]   [align=left] {$\displaystyle G_{1}$};
\draw (628,246) node [anchor=north west][inner sep=0.75pt]   [align=left] {$\displaystyle G_{2}$};
\draw (173,435) node [anchor=north west][inner sep=0.75pt]   [align=left] {$\displaystyle L( G_{1})$};
\draw (595,437) node [anchor=north west][inner sep=0.75pt]   [align=left] {$\displaystyle L( G_{2})$};

\end{tikzpicture}
}
\caption{A figure depicting the behavior of the defined $L\phi$. }
\end{figure}

However, notice that if $\phi(x) = \phi(z)$ then $L(\phi)(a) = L(\phi)(b)$, which means $L(\phi):L(G_1) \to L(G_2)$ is in fact only a 1 complex morphism. So the association $\phi \mapsto L(\phi)$ is well defined, and it is clear that it preserves composition, so we have defined a functor $L$ - but not exactly from the category of graphs to itself, but actually from the category of graphs with graph homomorphisms as morphism towards the category of graphs with 1-complex morphism as morphisms. 

We want to compose $L$ with the $\Omega$ functor, so we have to restrict our morphisms a little more.

\begin{proposition}
\label{linepropcharact}
    Let $\phi:G_1 \to G_2$ be a graph homomorphism. Then $L(\phi)$ is strongly proper if and only if $L(\phi)^{-1}(e)$ is finite for every edge $e \in \mathrm{E}(G_2)$. 
\end{proposition}
\begin{proof}
($\impliedby$) Suppose $L(\phi)$ isn't strongly proper. Then, there is an $e \in E(G_2) = V(L(G_2))$ such that $L(\phi)^{-1}(e)$ is not dispersed. Since finite sets are dispersed, this means $L(\phi)^{-1}(e)$ must be infinite.\\

($\implies$) Suppose there is an edge $e \in E(G_2) = V(L(G_2))$ such that $L(\phi)^{-1}(e)$ is infinite. This implies that the set $U \subset V(G_1) $ of vertices incident to the edges in $L(\phi)^{-1}(e)$ must be infinite too. Using the Star-Comb Lemma in $G_1$, there is either a star infinitely edge-connected to $U$ or a ray infinitely edge-connected to $U$. Let $\rho \in \Omega(L(G_1))$ be the end associated to either that star or that ray. In either case, if we were to remove from $L(G_1)$ a set $F \subset V(L(G_1))$, that would amount to removing from $G_1$ finitely many edges. This still leaves the star or the ray edge-connected to $U$, therefore $L(\phi)^{-1}(e)$ intersects the connected component of $L(G_1) \setminus F$ chosen by $\rho$. This means $L(\phi)^{-1}(e)$ is not dispersed, and therefore $L(\phi)$ is not strongly proper. 

\end{proof}

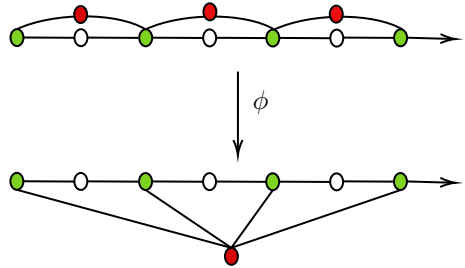
\begin{figure}[ht]
    \centering

\tikzset{every picture/.style={line width=0.75pt}} 

\begin{tikzpicture}[x=0.75pt,y=0.75pt,yscale=-1,xscale=1,scale = 0.7]

\draw  [fill={rgb, 255:red, 126; green, 211; blue, 33 }  ,fill opacity=1 ] (132.16,126.7) .. controls (132.11,123.35) and (134.16,120.59) .. (136.72,120.56) .. controls (139.28,120.52) and (141.39,123.21) .. (141.44,126.57) .. controls (141.48,129.93) and (139.44,132.68) .. (136.88,132.72) .. controls (134.31,132.75) and (132.2,130.06) .. (132.16,126.7) -- cycle ;
\draw   (177.99,126.78) .. controls (177.95,123.42) and (179.99,120.67) .. (182.55,120.63) .. controls (185.12,120.59) and (187.23,123.29) .. (187.27,126.64) .. controls (187.32,130) and (185.28,132.75) .. (182.71,132.79) .. controls (180.15,132.83) and (178.04,130.14) .. (177.99,126.78) -- cycle ;
\draw    (141.44,126.57) -- (178,126.78) ;
\draw    (187.27,126.64) -- (223.83,126.85) ;
\draw  [fill={rgb, 255:red, 126; green, 211; blue, 33 }  ,fill opacity=1 ] (224.4,126.85) .. controls (224.35,123.49) and (226.4,120.74) .. (228.96,120.7) .. controls (231.52,120.66) and (233.63,123.35) .. (233.68,126.71) .. controls (233.72,130.07) and (231.68,132.82) .. (229.12,132.86) .. controls (226.55,132.89) and (224.44,130.2) .. (224.4,126.85) -- cycle ;
\draw   (270.23,126.92) .. controls (270.19,123.56) and (272.23,120.81) .. (274.79,120.77) .. controls (277.36,120.74) and (279.47,123.43) .. (279.51,126.78) .. controls (279.56,130.14) and (277.52,132.89) .. (274.95,132.93) .. controls (272.39,132.97) and (270.28,130.28) .. (270.23,126.92) -- cycle ;
\draw    (233.68,126.71) -- (270.24,126.92) ;
\draw    (279.51,126.78) -- (316.07,126.99) ;
\draw  [fill={rgb, 255:red, 126; green, 211; blue, 33 }  ,fill opacity=1 ] (315.51,127) .. controls (315.46,123.65) and (317.5,120.89) .. (320.07,120.86) .. controls (322.63,120.82) and (324.74,123.51) .. (324.79,126.87) .. controls (324.83,130.22) and (322.79,132.98) .. (320.23,133.01) .. controls (317.66,133.05) and (315.55,130.36) .. (315.51,127) -- cycle ;
\draw   (361.34,127.08) .. controls (361.3,123.72) and (363.34,120.97) .. (365.9,120.93) .. controls (368.47,120.89) and (370.58,123.58) .. (370.62,126.94) .. controls (370.67,130.3) and (368.63,133.05) .. (366.06,133.09) .. controls (363.5,133.13) and (361.39,130.44) .. (361.34,127.08) -- cycle ;
\draw    (324.79,126.87) -- (361.34,127.08) ;
\draw    (136.72,120.56) .. controls (163.6,108.3) and (208.31,108.39) .. (228.96,120.7) ;
\draw    (228.96,120.7) .. controls (255.84,108.44) and (300.55,108.53) .. (321.2,120.84) ;
\draw    (321.2,120.84) .. controls (348.08,108.58) and (392.79,108.68) .. (413.44,120.98) ;
\draw  [fill={rgb, 255:red, 126; green, 211; blue, 33 }  ,fill opacity=1 ] (416.35,126.71) .. controls (416.4,130.07) and (414.36,132.82) .. (411.79,132.86) .. controls (409.23,132.89) and (407.12,130.2) .. (407.07,126.85) .. controls (407.03,123.49) and (409.07,120.74) .. (411.64,120.7) .. controls (414.2,120.66) and (416.31,123.35) .. (416.35,126.71) -- cycle ;
\draw    (407.08,126.84) -- (370.52,126.64) ;
\draw    (416.35,126.71) -- (449.2,127.63) ;
\draw [shift={(451.2,127.69)}, rotate = 181.6] [color={rgb, 255:red, 0; green, 0; blue, 0 }  ][line width=0.75pt]    (10.93,-3.29) .. controls (6.95,-1.4) and (3.31,-0.3) .. (0,0) .. controls (3.31,0.3) and (6.95,1.4) .. (10.93,3.29)   ;
\draw  [fill={rgb, 255:red, 217; green, 6; blue, 6 }  ,fill opacity=1 ] (177.78,110.18) .. controls (177.74,106.82) and (179.78,104.07) .. (182.34,104.03) .. controls (184.91,103.99) and (187.02,106.68) .. (187.06,110.04) .. controls (187.11,113.4) and (185.07,116.15) .. (182.5,116.19) .. controls (179.94,116.23) and (177.83,113.54) .. (177.78,110.18) -- cycle ;
\draw  [fill={rgb, 255:red, 232; green, 16; blue, 16 }  ,fill opacity=1 ] (270.43,108.24) .. controls (270.39,104.88) and (272.43,102.13) .. (274.99,102.09) .. controls (277.55,102.05) and (279.67,104.74) .. (279.71,108.1) .. controls (279.75,111.46) and (277.71,114.21) .. (275.15,114.25) .. controls (272.59,114.29) and (270.47,111.6) .. (270.43,108.24) -- cycle ;
\draw  [fill={rgb, 255:red, 235; green, 15; blue, 15 }  ,fill opacity=1 ] (360.98,109.88) .. controls (360.93,106.52) and (362.98,103.77) .. (365.54,103.73) .. controls (368.1,103.69) and (370.21,106.38) .. (370.26,109.74) .. controls (370.3,113.1) and (368.26,115.85) .. (365.7,115.89) .. controls (363.13,115.93) and (361.02,113.23) .. (360.98,109.88) -- cycle ;
\draw  [fill={rgb, 255:red, 126; green, 211; blue, 33 }  ,fill opacity=1 ] (132.07,229.72) .. controls (132.02,226.36) and (134.07,223.61) .. (136.63,223.57) .. controls (139.19,223.54) and (141.3,226.23) .. (141.35,229.59) .. controls (141.39,232.94) and (139.35,235.69) .. (136.79,235.73) .. controls (134.23,235.77) and (132.11,233.08) .. (132.07,229.72) -- cycle ;
\draw   (177.91,229.8) .. controls (177.86,226.44) and (179.9,223.69) .. (182.47,223.65) .. controls (185.03,223.61) and (187.14,226.3) .. (187.19,229.66) .. controls (187.23,233.02) and (185.19,235.77) .. (182.63,235.81) .. controls (180.06,235.84) and (177.95,233.15) .. (177.91,229.8) -- cycle ;
\draw    (141.35,229.59) -- (177.91,229.8) ;
\draw    (187.18,229.66) -- (223.74,229.87) ;
\draw  [fill={rgb, 255:red, 126; green, 211; blue, 33 }  ,fill opacity=1 ] (224.31,229.86) .. controls (224.26,226.5) and (226.31,223.75) .. (228.87,223.72) .. controls (231.43,223.68) and (233.54,226.37) .. (233.59,229.73) .. controls (233.63,233.08) and (231.59,235.84) .. (229.03,235.87) .. controls (226.47,235.91) and (224.35,233.22) .. (224.31,229.86) -- cycle ;
\draw   (270.15,229.94) .. controls (270.1,226.58) and (272.14,223.83) .. (274.71,223.79) .. controls (277.27,223.75) and (279.38,226.44) .. (279.43,229.8) .. controls (279.47,233.16) and (277.43,235.91) .. (274.87,235.95) .. controls (272.3,235.99) and (270.19,233.29) .. (270.15,229.94) -- cycle ;
\draw    (233.59,229.73) -- (270.15,229.94) ;
\draw    (279.42,229.8) -- (315.98,230.01) ;
\draw  [fill={rgb, 255:red, 126; green, 211; blue, 33 }  ,fill opacity=1 ] (315.42,230.02) .. controls (315.37,226.66) and (317.41,223.91) .. (319.98,223.87) .. controls (322.54,223.84) and (324.65,226.53) .. (324.7,229.88) .. controls (324.74,233.24) and (322.7,235.99) .. (320.14,236.03) .. controls (317.57,236.07) and (315.46,233.38) .. (315.42,230.02) -- cycle ;
\draw   (361.25,230.09) .. controls (361.21,226.74) and (363.25,223.98) .. (365.81,223.95) .. controls (368.38,223.91) and (370.49,226.6) .. (370.53,229.96) .. controls (370.58,233.32) and (368.54,236.07) .. (365.97,236.11) .. controls (363.41,236.14) and (361.3,233.45) .. (361.25,230.09) -- cycle ;
\draw    (324.7,229.88) -- (361.25,230.09) ;
\draw  [fill={rgb, 255:red, 126; green, 211; blue, 33 }  ,fill opacity=1 ] (416.27,229.73) .. controls (416.31,233.08) and (414.27,235.84) .. (411.71,235.87) .. controls (409.14,235.91) and (407.03,233.22) .. (406.99,229.86) .. controls (406.94,226.5) and (408.98,223.75) .. (411.55,223.71) .. controls (414.11,223.68) and (416.22,226.37) .. (416.27,229.73) -- cycle ;
\draw    (406.99,229.86) -- (370.43,229.65) ;
\draw    (416.27,229.73) -- (449.11,230.65) ;
\draw [shift={(451.11,230.7)}, rotate = 181.6] [color={rgb, 255:red, 0; green, 0; blue, 0 }  ][line width=0.75pt]    (10.93,-3.29) .. controls (6.95,-1.4) and (3.31,-0.3) .. (0,0) .. controls (3.31,0.3) and (6.95,1.4) .. (10.93,3.29)   ;
\draw  [fill={rgb, 255:red, 217; green, 6; blue, 6 }  ,fill opacity=1 ] (285.84,283.47) .. controls (285.79,280.11) and (287.84,277.36) .. (290.4,277.32) .. controls (292.96,277.28) and (295.07,279.97) .. (295.12,283.33) .. controls (295.16,286.69) and (293.12,289.44) .. (290.56,289.48) .. controls (287.99,289.52) and (285.88,286.82) .. (285.84,283.47) -- cycle ;
\draw    (295.16,150.98) -- (295.12,208.99) ;
\draw [shift={(295.12,210.99)}, rotate = 270.03] [color={rgb, 255:red, 0; green, 0; blue, 0 }  ][line width=0.75pt]    (10.93,-3.29) .. controls (6.95,-1.4) and (3.31,-0.3) .. (0,0) .. controls (3.31,0.3) and (6.95,1.4) .. (10.93,3.29)   ;
\draw    (136.79,235.73) -- (290.4,277.32) ;
\draw    (229.03,235.87) -- (290.4,277.32) ;
\draw    (320.14,236.03) -- (290.4,277.32) ;
\draw    (411.71,235.87) -- (290.4,277.32) ;

\draw (303.34,162.25) node [anchor=north west][inner sep=0.75pt]  [rotate=-359.08] [align=left] {$\displaystyle \phi $};

\end{tikzpicture}
\caption{\myref{linepropcharact} does not not imply that the fibers through $\phi$ must be finite. Consider the quotient collapsing the red vertices in the figure. Notice every edge still has a finite fiber, even though this collapsed infinitely many vertices.}
\end{figure}

\begin{definition}[Strongly line-proper]
A graph homomorphism $\phi:G_1 \rightarrow G_2$ is called \textbf{strongly line-proper} when one of the equivalent conditions of \myref{linepropcharact} occurs: $L(\phi)$ is strongly proper or $\phi^{-1}(e)$ is finite for every edge $e \in E(G_2)$.
\end{definition}

The strongly line-proper morphisms are the class of morphisms we will use. 

\begin{example}[A morphism that is not strongly line-proper]
\label{notstrongLP}
Consider the morphism described by the following picture, where red/green/black/blue/yellow vertices get taken to the unique red/green/black/blue/yellow vertex of the second graph; the white vertices of the ray on top get taken sequentially to the white vertices of the ray on top of the second graph, and the white vertices from the ray on the bottom get taken sequentially to the white vertices of the ray on the bottom.

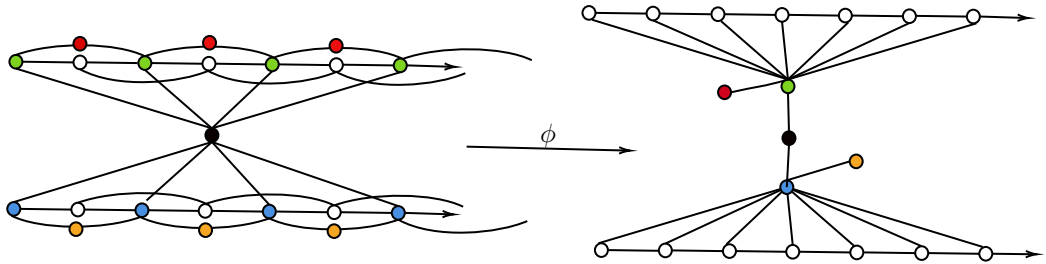
\begin{figure}[ht]
    \centering
    
\tikzset{every picture/.style={line width=0.75pt}} 

\begin{tikzpicture}[x=0.75pt,y=0.75pt,yscale=-1,xscale=1,scale=0.5]

\draw  [fill={rgb, 255:red, 126; green, 211; blue, 33 }  ,fill opacity=1 ] (8.44,142.42) .. controls (8.43,138.65) and (11.34,135.57) .. (14.93,135.57) .. controls (18.53,135.56) and (21.45,138.61) .. (21.46,142.39) .. controls (21.47,146.17) and (18.56,149.24) .. (14.97,149.25) .. controls (11.37,149.26) and (8.45,146.2) .. (8.44,142.42) -- cycle ;
\draw   (72.74,143.09) .. controls (72.73,139.31) and (75.64,136.24) .. (79.24,136.23) .. controls (82.83,136.23) and (85.75,139.28) .. (85.76,143.06) .. controls (85.77,146.84) and (82.86,149.91) .. (79.27,149.92) .. controls (75.67,149.93) and (72.75,146.87) .. (72.74,143.09) -- cycle ;
\draw    (21.46,142.39) -- (72.74,143.09) ;
\draw    (85.76,143.06) -- (137.04,143.76) ;
\draw  [fill={rgb, 255:red, 126; green, 211; blue, 33 }  ,fill opacity=1 ] (137.84,143.76) .. controls (137.83,139.98) and (140.74,136.91) .. (144.33,136.9) .. controls (147.93,136.89) and (150.85,139.95) .. (150.86,143.73) .. controls (150.87,147.5) and (147.96,150.57) .. (144.36,150.58) .. controls (140.77,150.59) and (137.85,147.54) .. (137.84,143.76) -- cycle ;
\draw   (202.14,144.43) .. controls (202.13,140.65) and (205.04,137.58) .. (208.63,137.57) .. controls (212.23,137.56) and (215.15,140.62) .. (215.16,144.39) .. controls (215.17,148.17) and (212.26,151.24) .. (208.66,151.25) .. controls (205.07,151.26) and (202.15,148.21) .. (202.14,144.43) -- cycle ;
\draw    (150.86,143.73) -- (202.14,144.43) ;
\draw    (215.16,144.4) -- (266.44,145.1) ;
\draw  [fill={rgb, 255:red, 126; green, 211; blue, 33 }  ,fill opacity=1 ] (265.65,145.1) .. controls (265.64,141.32) and (268.55,138.25) .. (272.14,138.24) .. controls (275.74,138.23) and (278.66,141.29) .. (278.67,145.07) .. controls (278.68,148.84) and (275.77,151.91) .. (272.17,151.92) .. controls (268.58,151.93) and (265.66,148.88) .. (265.65,145.1) -- cycle ;
\draw   (329.95,145.77) .. controls (329.94,141.99) and (332.85,138.92) .. (336.44,138.91) .. controls (340.04,138.9) and (342.96,141.96) .. (342.97,145.74) .. controls (342.98,149.51) and (340.07,152.58) .. (336.47,152.59) .. controls (332.88,152.6) and (329.96,149.55) .. (329.95,145.77) -- cycle ;
\draw    (278.67,145.07) -- (329.95,145.77) ;
\draw    (14.93,135.57) .. controls (52.85,122.12) and (115.56,122.79) .. (144.33,136.9) ;
\draw    (144.33,136.9) .. controls (182.25,123.46) and (244.96,124.13) .. (273.73,138.24) ;
\draw    (273.73,138.24) .. controls (311.64,124.79) and (374.36,125.47) .. (403.13,139.57) ;
\draw  [fill={rgb, 255:red, 74; green, 144; blue, 226 }  ,fill opacity=1 ] (405.13,294.06) .. controls (405.14,297.84) and (402.23,300.91) .. (398.63,300.92) .. controls (395.04,300.92) and (392.12,297.87) .. (392.11,294.09) .. controls (392.1,290.31) and (395.01,287.24) .. (398.6,287.23) .. controls (402.2,287.22) and (405.12,290.28) .. (405.13,294.06) -- cycle ;
\draw   (340.83,293.39) .. controls (340.83,297.17) and (337.93,300.24) .. (334.33,300.25) .. controls (330.74,300.26) and (327.82,297.2) .. (327.81,293.42) .. controls (327.8,289.64) and (330.71,286.57) .. (334.3,286.56) .. controls (337.9,286.55) and (340.82,289.61) .. (340.83,293.39) -- cycle ;
\draw    (392.11,294.09) -- (340.83,293.39) ;
\draw    (327.81,293.42) -- (276.52,292.72) ;
\draw  [fill={rgb, 255:red, 74; green, 144; blue, 226 }  ,fill opacity=1 ] (275.73,292.72) .. controls (275.74,296.5) and (272.83,299.57) .. (269.24,299.58) .. controls (265.64,299.59) and (262.72,296.53) .. (262.71,292.75) .. controls (262.7,288.98) and (265.61,285.91) .. (269.21,285.9) .. controls (272.8,285.89) and (275.72,288.94) .. (275.73,292.72) -- cycle ;
\draw   (211.43,292.05) .. controls (211.44,295.83) and (208.53,298.9) .. (204.93,298.91) .. controls (201.34,298.92) and (198.42,295.86) .. (198.41,292.09) .. controls (198.4,288.31) and (201.31,285.24) .. (204.9,285.23) .. controls (208.5,285.22) and (211.42,288.27) .. (211.43,292.05) -- cycle ;
\draw    (262.71,292.75) -- (211.43,292.05) ;
\draw    (198.41,292.09) -- (147.13,291.38) ;
\draw  [fill={rgb, 255:red, 74; green, 144; blue, 226 }  ,fill opacity=1 ] (147.92,291.38) .. controls (147.93,295.16) and (145.02,298.23) .. (141.43,298.24) .. controls (137.83,298.25) and (134.91,295.19) .. (134.9,291.41) .. controls (134.89,287.64) and (137.8,284.57) .. (141.4,284.56) .. controls (144.99,284.55) and (147.91,287.6) .. (147.92,291.38) -- cycle ;
\draw   (83.62,290.71) .. controls (83.63,294.49) and (80.72,297.56) .. (77.12,297.57) .. controls (73.53,297.58) and (70.61,294.52) .. (70.6,290.75) .. controls (70.59,286.97) and (73.5,283.9) .. (77.09,283.89) .. controls (80.69,283.88) and (83.61,286.93) .. (83.62,290.71) -- cycle ;
\draw    (134.9,291.41) -- (83.62,290.71) ;
\draw    (405.13,294.06) -- (452,295.54) ;
\draw [shift={(454,295.6)}, rotate = 181.81] [color={rgb, 255:red, 0; green, 0; blue, 0 }  ][line width=0.75]    (10.93,-3.29) .. controls (6.95,-1.4) and (3.31,-0.3) .. (0,0) .. controls (3.31,0.3) and (6.95,1.4) .. (10.93,3.29)   ;
\draw    (398.63,300.91) .. controls (360.72,314.36) and (298.01,313.69) .. (269.24,299.58) ;
\draw    (269.24,299.58) .. controls (231.32,313.02) and (168.61,312.35) .. (139.84,298.24) ;
\draw    (139.84,298.24) .. controls (101.92,311.69) and (39.21,311.02) .. (10.44,296.91) ;
\draw    (19.32,290.04) -- (70.6,290.75) ;
\draw  [fill={rgb, 255:red, 74; green, 144; blue, 226 }  ,fill opacity=1 ] (6.3,290.08) .. controls (6.29,286.3) and (9.2,283.23) .. (12.79,283.22) .. controls (16.39,283.21) and (19.31,286.26) .. (19.32,290.04) .. controls (19.32,293.82) and (16.42,296.89) .. (12.82,296.9) .. controls (9.23,296.91) and (6.31,293.86) .. (6.3,290.08) -- cycle ;
\draw  [fill={rgb, 255:red, 126; green, 211; blue, 33 }  ,fill opacity=1 ] (407.13,146.06) .. controls (407.14,149.84) and (404.23,152.91) .. (400.63,152.92) .. controls (397.04,152.92) and (394.12,149.87) .. (394.11,146.09) .. controls (394.1,142.31) and (397.01,139.24) .. (400.6,139.23) .. controls (404.2,139.22) and (407.12,142.28) .. (407.13,146.06) -- cycle ;
\draw    (394.11,146.09) -- (342.83,145.39) ;
\draw    (982.13,97.06) -- (1029,98.54) ;
\draw [shift={(1031,98.6)}, rotate = 181.81] [color={rgb, 255:red, 0; green, 0; blue, 0 }  ][line width=0.75]    (10.93,-3.29) .. controls (6.95,-1.4) and (3.31,-0.3) .. (0,0) .. controls (3.31,0.3) and (6.95,1.4) .. (10.93,3.29)   ;
\draw  [fill={rgb, 255:red, 11; green, 1; blue, 1 }  ,fill opacity=1 ] (218.13,216.06) .. controls (218.14,219.84) and (215.23,222.91) .. (211.63,222.92) .. controls (208.04,222.92) and (205.12,219.87) .. (205.11,216.09) .. controls (205.1,212.31) and (208.01,209.24) .. (211.6,209.23) .. controls (215.2,209.22) and (218.12,212.28) .. (218.13,216.06) -- cycle ;
\draw    (14.97,149.25) -- (211.6,209.23) ;
\draw    (144.36,150.58) -- (211.6,209.23) ;
\draw    (272.17,151.92) -- (211.6,209.23) ;
\draw    (400.63,152.92) -- (211.6,209.23) ;
\draw    (12.79,283.22) -- (211.63,222.92) ;
\draw    (211.63,222.92) -- (146,281.6) ;
\draw    (211.63,222.92) -- (269.21,285.9) ;
\draw    (211.63,222.92) -- (398.6,287.23) ;
\draw   (583.44,93.42) .. controls (583.43,89.65) and (586.34,86.57) .. (589.93,86.57) .. controls (593.53,86.56) and (596.45,89.61) .. (596.46,93.39) .. controls (596.47,97.17) and (593.56,100.24) .. (589.97,100.25) .. controls (586.37,100.26) and (583.45,97.2) .. (583.44,93.42) -- cycle ;
\draw   (647.74,94.09) .. controls (647.73,90.31) and (650.64,87.24) .. (654.24,87.23) .. controls (657.83,87.23) and (660.75,90.28) .. (660.76,94.06) .. controls (660.77,97.84) and (657.86,100.91) .. (654.27,100.92) .. controls (650.67,100.93) and (647.75,97.87) .. (647.74,94.09) -- cycle ;
\draw    (596.46,93.39) -- (647.74,94.09) ;
\draw    (660.76,94.06) -- (712.04,94.76) ;
\draw   (712.84,94.76) .. controls (712.83,90.98) and (715.74,87.91) .. (719.33,87.9) .. controls (722.93,87.89) and (725.85,90.95) .. (725.86,94.73) .. controls (725.87,98.5) and (722.96,101.57) .. (719.36,101.58) .. controls (715.77,101.59) and (712.85,98.54) .. (712.84,94.76) -- cycle ;
\draw   (777.14,95.43) .. controls (777.13,91.65) and (780.04,88.58) .. (783.63,88.57) .. controls (787.23,88.56) and (790.15,91.62) .. (790.16,95.39) .. controls (790.17,99.17) and (787.26,102.24) .. (783.66,102.25) .. controls (780.07,102.26) and (777.15,99.21) .. (777.14,95.43) -- cycle ;
\draw    (725.86,94.73) -- (777.14,95.43) ;
\draw    (790.16,95.4) -- (841.44,96.1) ;
\draw   (840.65,96.1) .. controls (840.64,92.32) and (843.55,89.25) .. (847.14,89.24) .. controls (850.74,89.23) and (853.66,92.29) .. (853.67,96.07) .. controls (853.68,99.84) and (850.77,102.91) .. (847.17,102.92) .. controls (843.58,102.93) and (840.66,99.88) .. (840.65,96.1) -- cycle ;
\draw   (904.95,96.77) .. controls (904.94,92.99) and (907.85,89.92) .. (911.44,89.91) .. controls (915.04,89.9) and (917.96,92.96) .. (917.97,96.74) .. controls (917.98,100.51) and (915.07,103.58) .. (911.47,103.59) .. controls (907.88,103.6) and (904.96,100.55) .. (904.95,96.77) -- cycle ;
\draw    (853.67,96.07) -- (904.95,96.77) ;
\draw   (982.13,97.06) .. controls (982.14,100.84) and (979.23,103.91) .. (975.63,103.92) .. controls (972.04,103.92) and (969.12,100.87) .. (969.11,97.09) .. controls (969.1,93.31) and (972.01,90.24) .. (975.6,90.23) .. controls (979.2,90.22) and (982.12,93.28) .. (982.13,97.06) -- cycle ;
\draw    (969.11,97.09) -- (917.83,96.39) ;
\draw  [fill={rgb, 255:red, 126; green, 211; blue, 33 }  ,fill opacity=1 ] (796.13,167.06) .. controls (796.14,170.84) and (793.23,173.91) .. (789.63,173.92) .. controls (786.04,173.92) and (783.12,170.87) .. (783.11,167.09) .. controls (783.1,163.31) and (786.01,160.24) .. (789.6,160.23) .. controls (793.2,160.22) and (796.12,163.28) .. (796.13,167.06) -- cycle ;
\draw    (592.97,100.25) -- (789.6,160.23) ;
\draw    (722.36,101.58) -- (789.6,160.23) ;
\draw    (850.17,102.92) -- (789.6,160.23) ;
\draw    (978.63,103.92) -- (789.6,160.23) ;
\draw    (654.27,100.92) -- (789.6,160.23) ;
\draw    (783.66,102.25) -- (789.6,160.23) ;
\draw    (911.47,103.59) -- (789.6,160.23) ;
\draw    (407.13,146.06) -- (454,147.54) ;
\draw [shift={(456,147.6)}, rotate = 181.81] [color={rgb, 255:red, 0; green, 0; blue, 0 }  ][line width=0.75]    (10.93,-3.29) .. controls (6.95,-1.4) and (3.31,-0.3) .. (0,0) .. controls (3.31,0.3) and (6.95,1.4) .. (10.93,3.29)   ;
\draw  [fill={rgb, 255:red, 217; green, 6; blue, 6 }  ,fill opacity=1 ] (72.71,124.41) .. controls (72.7,120.63) and (75.61,117.56) .. (79.21,117.55) .. controls (82.8,117.54) and (85.72,120.6) .. (85.73,124.38) .. controls (85.74,128.15) and (82.83,131.23) .. (79.24,131.23) .. controls (75.64,131.24) and (72.72,128.19) .. (72.71,124.41) -- cycle ;
\draw  [fill={rgb, 255:red, 232; green, 16; blue, 16 }  ,fill opacity=1 ] (202.71,123.41) .. controls (202.7,119.63) and (205.61,116.56) .. (209.21,116.55) .. controls (212.8,116.54) and (215.72,119.6) .. (215.73,123.38) .. controls (215.74,127.15) and (212.83,130.23) .. (209.24,130.23) .. controls (205.64,130.24) and (202.72,127.19) .. (202.71,123.41) -- cycle ;
\draw  [fill={rgb, 255:red, 235; green, 15; blue, 15 }  ,fill opacity=1 ] (329.71,126.41) .. controls (329.7,122.63) and (332.61,119.56) .. (336.21,119.55) .. controls (339.8,119.54) and (342.72,122.6) .. (342.73,126.38) .. controls (342.74,130.15) and (339.83,133.23) .. (336.24,133.23) .. controls (332.64,133.24) and (329.72,130.19) .. (329.71,126.41) -- cycle ;
\draw  [fill={rgb, 255:red, 245; green, 166; blue, 35 }  ,fill opacity=1 ] (68.71,310.41) .. controls (68.7,306.63) and (71.61,303.56) .. (75.21,303.55) .. controls (78.8,303.54) and (81.72,306.6) .. (81.73,310.38) .. controls (81.74,314.15) and (78.83,317.23) .. (75.24,317.23) .. controls (71.64,317.24) and (68.72,314.19) .. (68.71,310.41) -- cycle ;
\draw  [fill={rgb, 255:red, 245; green, 166; blue, 35 }  ,fill opacity=1 ] (198.71,311.41) .. controls (198.7,307.63) and (201.61,304.56) .. (205.21,304.55) .. controls (208.8,304.54) and (211.72,307.6) .. (211.73,311.38) .. controls (211.74,315.15) and (208.83,318.23) .. (205.24,318.23) .. controls (201.64,318.24) and (198.72,315.19) .. (198.71,311.41) -- cycle ;
\draw  [fill={rgb, 255:red, 245; green, 166; blue, 35 }  ,fill opacity=1 ] (327.71,312.41) .. controls (327.7,308.63) and (330.61,305.56) .. (334.21,305.55) .. controls (337.8,305.54) and (340.72,308.6) .. (340.73,312.38) .. controls (340.74,316.15) and (337.83,319.23) .. (334.24,319.23) .. controls (330.64,319.24) and (327.72,316.19) .. (327.71,312.41) -- cycle ;
\draw  [fill={rgb, 255:red, 208; green, 2; blue, 27 }  ,fill opacity=1 ] (719.47,173.11) .. controls (719.46,169.33) and (722.37,166.26) .. (725.97,166.25) .. controls (729.56,166.24) and (732.48,169.29) .. (732.49,173.07) .. controls (732.5,176.85) and (729.59,179.92) .. (726,179.93) .. controls (722.4,179.94) and (719.48,176.88) .. (719.47,173.11) -- cycle ;
\draw  [fill={rgb, 255:red, 11; green, 1; blue, 1 }  ,fill opacity=1 ] (797.13,219.06) .. controls (797.14,222.84) and (794.23,225.91) .. (790.63,225.92) .. controls (787.04,225.92) and (784.12,222.87) .. (784.11,219.09) .. controls (784.1,215.31) and (787.01,212.24) .. (790.6,212.23) .. controls (794.2,212.22) and (797.12,215.28) .. (797.13,219.06) -- cycle ;
\draw    (789.63,173.92) -- (790.6,212.23) ;
\draw    (994.6,335.06) -- (1037.2,335.58) ;
\draw [shift={(1039.2,335.6)}, rotate = 180.7] [color={rgb, 255:red, 0; green, 0; blue, 0 }  ][line width=0.75]    (10.93,-3.29) .. controls (6.95,-1.4) and (3.31,-0.3) .. (0,0) .. controls (3.31,0.3) and (6.95,1.4) .. (10.93,3.29)   ;
\draw   (994.6,335.06) .. controls (994.61,338.84) and (991.7,341.91) .. (988.11,341.92) .. controls (984.51,341.92) and (981.59,338.87) .. (981.58,335.09) .. controls (981.57,331.31) and (984.48,328.24) .. (988.07,328.23) .. controls (991.67,328.22) and (994.59,331.28) .. (994.6,335.06) -- cycle ;
\draw   (930.3,334.39) .. controls (930.31,338.17) and (927.4,341.24) .. (923.8,341.25) .. controls (920.21,341.26) and (917.29,338.2) .. (917.28,334.42) .. controls (917.27,330.64) and (920.18,327.57) .. (923.77,327.56) .. controls (927.37,327.55) and (930.29,330.61) .. (930.3,334.39) -- cycle ;
\draw    (981.58,335.09) -- (930.3,334.39) ;
\draw    (917.28,334.42) -- (866,333.72) ;
\draw   (865.2,333.72) .. controls (865.21,337.5) and (862.3,340.57) .. (858.71,340.58) .. controls (855.11,340.59) and (852.19,337.53) .. (852.18,333.75) .. controls (852.17,329.98) and (855.08,326.91) .. (858.68,326.9) .. controls (862.27,326.89) and (865.19,329.94) .. (865.2,333.72) -- cycle ;
\draw   (800.9,333.05) .. controls (800.91,336.83) and (798,339.9) .. (794.41,339.91) .. controls (790.81,339.92) and (787.89,336.86) .. (787.88,333.09) .. controls (787.87,329.31) and (790.78,326.24) .. (794.38,326.23) .. controls (797.97,326.22) and (800.89,329.27) .. (800.9,333.05) -- cycle ;
\draw    (852.18,333.75) -- (800.9,333.05) ;
\draw    (787.88,333.09) -- (736.6,332.38) ;
\draw   (737.39,332.38) .. controls (737.4,336.16) and (734.49,339.23) .. (730.9,339.24) .. controls (727.3,339.25) and (724.38,336.19) .. (724.37,332.41) .. controls (724.36,328.64) and (727.27,325.57) .. (730.87,325.56) .. controls (734.46,325.55) and (737.38,328.6) .. (737.39,332.38) -- cycle ;
\draw   (673.09,331.71) .. controls (673.1,335.49) and (670.19,338.56) .. (666.6,338.57) .. controls (663,338.58) and (660.08,335.52) .. (660.07,331.75) .. controls (660.06,327.97) and (662.97,324.9) .. (666.57,324.89) .. controls (670.16,324.88) and (673.08,327.93) .. (673.09,331.71) -- cycle ;
\draw    (724.37,332.41) -- (673.09,331.71) ;
\draw   (595.91,331.42) .. controls (595.9,327.65) and (598.81,324.57) .. (602.41,324.57) .. controls (606,324.56) and (608.92,327.61) .. (608.93,331.39) .. controls (608.94,335.17) and (606.03,338.24) .. (602.44,338.25) .. controls (598.84,338.26) and (595.92,335.2) .. (595.91,331.42) -- cycle ;
\draw    (608.93,331.39) -- (660.21,332.09) ;
\draw    (985.07,328.23) -- (788.44,268.25) ;
\draw    (855.68,326.9) -- (788.44,268.25) ;
\draw    (727.87,325.56) -- (788.44,268.25) ;
\draw    (599.41,324.57) -- (788.44,268.25) ;
\draw    (923.77,327.56) -- (788.44,268.25) ;
\draw    (794.38,326.23) -- (788.44,268.25) ;
\draw    (666.57,324.89) -- (788.44,268.25) ;
\draw  [fill={rgb, 255:red, 245; green, 166; blue, 35 }  ,fill opacity=1 ] (864.57,242.37) .. controls (864.58,246.15) and (861.67,249.22) .. (858.07,249.23) .. controls (854.48,249.24) and (851.56,246.19) .. (851.55,242.41) .. controls (851.54,238.63) and (854.45,235.56) .. (858.04,235.55) .. controls (861.64,235.54) and (864.56,238.6) .. (864.57,242.37) -- cycle ;
\draw  [fill={rgb, 255:red, 74; green, 144; blue, 226 }  ,fill opacity=1 ] (794.95,268.23) .. controls (794.96,272.01) and (792.05,275.08) .. (788.45,275.09) .. controls (784.86,275.1) and (781.94,272.04) .. (781.93,268.26) .. controls (781.92,264.49) and (784.83,261.42) .. (788.42,261.41) .. controls (792.02,261.4) and (794.94,264.45) .. (794.95,268.23) -- cycle ;
\draw    (790.63,225.92) -- (788.44,268.25) ;
\draw    (789.6,160.23) .. controls (765.2,167.6) and (756.2,167.6) .. (732.49,173.07) ;
\draw    (467.2,227.6) -- (629.2,231.55) ;
\draw [shift={(631.2,231.6)}, rotate = 181.4] [color={rgb, 255:red, 0; green, 0; blue, 0 }  ][line width=0.75]    (10.93,-3.29) .. controls (6.95,-1.4) and (3.31,-0.3) .. (0,0) .. controls (3.31,0.3) and (6.95,1.4) .. (10.93,3.29)   ;
\draw    (79.27,149.92) .. controls (116.5,170.27) and (180.2,163.6) .. (208.66,151.25) ;
\draw    (207.08,151.26) .. controls (244.31,171.61) and (308.01,164.94) .. (336.47,152.59) ;
\draw    (77.09,283.89) .. controls (115.01,270.44) and (177.72,271.12) .. (206.49,285.22) ;
\draw    (204.9,285.23) .. controls (242.82,271.78) and (305.53,272.46) .. (334.3,286.56) ;
\draw    (334.3,286.56) .. controls (372.21,273.12) and (434.93,273.79) .. (463.7,287.9) ;
\draw    (336.47,152.59) .. controls (373.71,172.95) and (437.41,166.28) .. (465.87,153.93) ;
\draw    (403.13,139.57) .. controls (441.04,126.13) and (503.76,126.8) .. (532.53,140.91) ;
\draw    (398.63,300.91) .. controls (435.87,321.27) and (499.57,314.6) .. (528.03,302.25) ;
\draw    (788.42,261.41) -- (851.55,242.41) ;

\draw (538,201) node [anchor=north west][inner sep=0.75pt]   [align=left] {$\displaystyle \phi $};

\end{tikzpicture}
\end{figure}

Note the edge joining together the black and green vertex on the second graph and the edge joining together the black and blue vertex on the second graph. Those edges get mapped into an infinite number of times, so $\phi$ cannot  be strongly line-proper. Notice that the infinite degree black vertex gets mapped into a finite degree black vertex - and that $L(G_1)$ has a single end, and $L(G_2)$ has two ends. So the (unique) end of $L(G_1)$ associated with the infinite degree black vertex has no natural candidate in $\Omega(L(G_2))$ to be mapped into: the map $L(\phi)$ cannot induce a continuous mapping between $\Omega(L(G_1))$ and $\Omega(L(G_2))$.

\end{example}

The union of finite sets is still a finite set, so the composition of strongly line-proper morphisms is still strongly line-proper. So we can conclude

\begin{theorem}[The line-proper category]
There is a category $\catname{Graph}_{LP}$ with objects as connected graphs and morphisms as the strongly line-proper morphisms. There is a functor $L: \catname{Graph}_{LP} \rightarrow \catname{Graph}_{SP}$ that maps graphs $G$ to their line graphs $L(G)$ and strongly line-proper morphisms $\phi: G_1 \rightarrow G_2$ to strongly proper morphisms $L(\phi): L(G_1) \rightarrow L(G_2)$, such that $L(\phi)(\{u,v\}) = \{\phi(u), \phi(v)\}$.
\end{theorem}

\subsection{Making \texorpdfstring{$\mathbb{\cutalgebrasymb}$}{Lg} into a functor}

Any map $\phi:G_1 \to G_2$ is characterized by a map between the sets of vertices of each of the graphs. So there is a natural Boolean algebra homomorphism $\mathcal{P}(\phi):\mathcal{P}({\mathrm{V}(G_2)}) \to \mathcal{P}({\mathrm{V}(G_2)})$ defined by $\mathcal{P}(\phi)(A) \doteq \phi^{-1}(A)$. We want to know when this induced homomorphism takes finite cuts to finite cuts. Reminding the reader, to define $\cutalgebra{G}$ we first start with the sets that form finite cuts \[\precut{G} \doteq \{A \subset \mathrm{V}(G) \st E(A,G \backslash A) \text{ is finite }\}\]

and then we take the quotient $\cutalgebra{G} \doteq \precut{G}/I_G$ where $I_G$ is the ideal of $\precut{G}$ consisting of finite sets of finite degree vertices of $G$.

\begin{definition}[Cut-proper morphisms]
    We say a graph homomorphism $\phi:G_1 \to G_2$ is \textbf{cut-proper} when $\phi^{-1}:\cutalgebrasymb(G_2) \to \cutalgebrasymb(G_1)$ is well-defined. These will be precisely the graph homomorphisms that satisfy
    \begin{itemize}
        \item[(i)] If $A \subset V(G_2)$ is connected to its complement $G_2 \setminus A$ via finitely many edges, then $\phi^{-1}(A)$ is connected to its complement $\phi^{-1}(G_2 \setminus A)$ via finitely many edges.
        \item[(ii)] If $v \in V(G_2)$ is of finite degree, then $\phi^{-1}(v)$ is a finite set of vertices of finite degree. This is so that $\phi^{-1}$ maps $I_{G_2}$ into $I_{G_1}$ and thus $\phi^{-1}$  is well defined in the quotient.

    \end{itemize}
\end{definition}


Thankfully, we don't need to restrict our morphism class any further - the strongly line-proper morphisms we described last time will do.\\
\begin{proposition}[Strongly line-proper implies cut-proper]
    Let $G_1,G_2$ be graphs, $G_1$ connected and $\phi:G_1 \to G_2$ a graph homomorphism. If $\phi$ is strongly line-proper, then it is cut-proper.
\end{proposition}
\begin{proof}
Take $A \subset V(G_2)$ connected to its complement via finitely many edges and let's call those edges $F \subset E(G_2)$. Any edge $\{x,y\}$ between $\phi^{-1}(A)$ and its complement must be mapped by $\phi$ into an edge $\{\phi(x), \phi(y)\} \in F$. Therefore the set of edges $F' \subset E(G_1)$ connecting $\phi^{-1}(A)$ and its complement must be such that $\ F' \subset L(\phi)^{-1}(F)$ - which must be a finite set, due to $\phi$ being strongly line-proper.\\
Take $v \in V(G_2)$ of finite degree. Call $N \subset E(G_2)$ the finite set of edges incident to $v$. A vertex $w \in V(G_1)$ mapped into $v$ by $\phi$ must be such that any edge incident to it gets mapped into $N$. Therefore, all the edges incident to $\phi^{-1}(v)$ must be in $L(\phi)^{-1}(N)$ - a finite set, due to $\phi$ being strongly line-proper. Therefore $\phi(v)$ contains only finite degree vertices, and a finite number of those - since $G_1$ is connected and therefore every vertex must have at least one edge incident to it.
\end{proof}

\myref{notstrongLP} provides an example of a map that is not strongly line-proper, and it is also not cut-proper: the pre-image of the cut induced by the black vertex on the second graph is not a finite cut. It would be very nice if these two classes of morphisms were equivalent - but they are not.

\begin{example}[A cut-proper morphism that is not strongly line-proper]
    Let $r$ be the graph consisting of a single ray $\sequence{w_1,w_2, w_3, \dots}$ and $H$ be the graph consisting of two vertices $x,y$ joined by a single edge $e = \{x,y\}$ and by infinitely many disjoint paths $\{\sequence{x,v_i,y}\}_{i \in \mathbb{N}}$. Consider the homomorphism $\phi:r \to H$ that maps $w_i$ to the $i$-th element of the following sequence \[\sequence{x,v_1,y,x,v_2,y,x,v_3,y,\dots}\]

    \begin{itemize}
        \item[(i)] \textbf{This morphism is cut-proper:} First note that a set $A \subset V(H)$ is only able to be in $\precut{H}$ if it or its complement consists of a finite number of the $v_i$. This implies that $\precut{H}$ is the union of $I_H$ and its complement, and so the quotient $\cutalgebra{H}$ is the trivial Boolean algebra $\{0,1\}$. The fact that the $\phi^{-1}$ of a finite number of $v_i$ is a finite number of $w_i$ is then enough for us to conclude that morphism will be cut-proper.
        \item[(ii)] \textbf{This morphism is not strongly line-proper:} The edge $e$ joining $x$ and $y$ gets mapped into by an infinite number of edges in the ray - this means that the set $\phi^{-1}(e)$ is not line-dispersed, since the edge-direction corresponding to the ray cannot be separated from $\phi^{-1}(e)$ by finitely many edges.
    \end{itemize}
\end{example}

It's clear that the composition of cut-proper morphisms is still a cut proper morphism. So, we can conclude, in summary

\begin{theorem}[The cut-proper category]
    There is a category $\catname{Graph}_{CP}$ with objects as the connected graphs and morphisms as the cut-proper morphisms. There is a (contravariant) functor  $\cutalgebrasymb: \catname{Graph}_{CP} \rightarrow \catname{Bool}$ taking a connected graph $G$ to its cut algebra $\cutalgebra{G}$ and a cut-proper morphism $\phi:G_1 \rightarrow G_2$ to a Boolean algebra morphism $\cutalgebrasymb(\phi): \cutalgebra{G_2} \rightarrow \cutalgebra{G_1}$ such that given an $A \in \cutalgebra{G_2}$ we have $\cutalgebra{\phi}(A) = \phi^{-1}(A) \in \cutalgebra{G_1}$.\\
    The same functor $\cutalgebrasymb$ is also well defined from the category $\catname{Graph}_{LP}$ of graphs with strongly line-proper morphisms to the category of Boolean algebras $\catname{Bool}$.
\end{theorem}

\subsection{Putting it all together: the edge-direction functor \texorpdfstring{$\mathcal{D}_E$}{Lg}}

With what we have so far, we can build two maps from the class of strongly line proper morphisms to the class of continuous maps: $\mathrm{Ult} \circ \cutalgebrasymb$ and $\Omega \circ L$. Let's prove that these are in fact the same. Remember that, from \myref{object-commute}, for any graph $G$ we have $\mathrm{Ult} \circ \cutalgebra{G} \cong \Omega \circ L (G) \cong \mathcal{D}_E(G)$. We will switch freely between describing elements of this space as ultrafilters, equivalence class of rays or edge-directions. Also remember that, regardless of the description, all of them are characterized by either an infinite degree vertex $v \in V(G)$ or a ray $r \subset V(G)$.

\begin{proposition}
Let $\phi: G_1 \rightarrow G_2$ be a strongly line-proper morphism. Then $\mathrm{Ult} \circ \cutalgebrasymb (\phi)= \Omega \circ L (\phi)$. Call this map $\mathcal{D}_E(\phi): \mathcal{D}_E(G_1) \rightarrow \mathcal{D}_E(G_2)$. Explicitly, $\mathcal{D}_E(\phi)$ maps a direction $\rho_v \in \mathcal{D}_E(G_1)$ given by an infinite degree vertex $v \in V(G_1)$ to the direction $\rho_{\phi(v)} \in \mathcal{D}_E(G_2)$ given by the infinite degree vertex $\phi(v) \in V(G_2)$; and it maps a direction $\rho_r \in \mathcal{D}_E(G_1)$ given by a ray $r \subset V(G_1)$ to the unique direction $\rho_{\phi(r)} \in \mathcal{D}_E(G_2)$ that always picks a connected component intersecting $\phi(r) \subset V(G_2)$.
\end{proposition}
\begin{proof}
Take $\rho \in \mathcal{D}_E(G_1)$ given by an infinite degree vertex $v \in V(G_1)$. Since $\phi$ is strongly line-proper, it must be cut-proper, so it must send $v$ to a $\phi(w) \in V(G_2)$ that is also of infinite degree. The set of edges adjacent to $v$, call it $r_v \subset E(G_1) = V(L(G_1))$ is a ray that represents the end $\rho$. By \myref{end-spacefunc}, $\Omega(L(\phi))(\rho)$ must be the unique end that cannot be finitely separated via vertices of $L(G_2)$ from $L(\phi)(r_v)$. This implies that the edge-direction induced by $\Omega(L(\phi))(\rho)$ is the one that always picks the vertex $\phi(v)$ - in other words, $\Omega(L(\phi))(\rho) = \rho_{\phi(v)}$.\\
On the ultrafilter side, the direction $\rho$ is represented by $u_v$, the principal ultrafilter induced by $v$. So we will have $\mathrm{Ult} \circ \cutalgebrasymb (\phi)(u_v) = \{A \in \cutalgebra{G_2} \st \phi^{-1}(A) \in u_v\} = \{A \in \cutalgebra{G_2} \st \phi(v) \in A \}= u_{\phi(v)}$, which is the ultrafilter corresponding to the direction $\rho_{\phi(v)}$. So, indeed, $\Omega(L(\phi))(\rho) = \mathrm{Ult} \circ \cutalgebrasymb (\phi)(\rho)$.\\
Now suppose $\rho \in \mathcal{D}_E(G_1)$ is given by a ray $r \subset V(G_1)$. Again by \myref{end-spacefunc}, $\Omega(L(\phi))(\rho)$ must be the unique end that cannot be finitely separated via vertices of $L(G_2)$ from $L(\phi)(r)$. Any direction in $\mathcal{D}_E(G_2)$ that always picks a connected component that intersects $\phi(r)$ must correspond to an end in $\Omega(L(\phi))(G_2)$ that cannot be finitely separated via vertices of $L(G_2)$ from $L(\phi)(r)$. So there is only one direction that always picks a connected component that intersects $\phi(r)$, we can call it $\rho_{\phi(r)}$ and we have $\Omega(L(\phi))(\rho) = \rho_{\phi(r)}$.\\
Call $u_r = \{A \in \cutalgebra{G_1} \st r \subset A\}$ the ultrafilter corresponding to the direction $\rho$. Note that, since $\cutalgebra{G_1}$ is a quotient, the inclusion as we have used is `modulo' a finite number of finite degree vertices - since $r$ is infinite, this ultrafilter is well defined. Now, $\mathrm{Ult} \circ \cutalgebrasymb (\phi)(u_r) = \{A \in \cutalgebra{G_2} \st \phi^{-1}(A) \in u_r\} = \{A \in \cutalgebra{G_2} \st \phi(r) \subset A \}$. This ultrafilter always intersects $\phi(r)$, so it defines a direction that always picks a connected component intersecting $\phi(r)$ - so it corresponds to $\rho_{\phi(r)}$, finalizing the proof that $\Omega(L(\phi))(\rho) = \mathrm{Ult} \circ \cutalgebrasymb (\phi)(\rho)$.

\end{proof}

In summary then: we have built, in two ways, a functor $\mathcal{D}_E$. We can define a category $\catname{Graph}_{CP}$ with objects as connected graphs, and morphisms being the cut-proper morphisms. Here we can define the functor $\cutalgebrasymb$ and therefore the functor $\mathcal{D}_E = \mathrm{Ult} \circ \cutalgebrasymb: \catname{Graph}_{CP} \rightarrow \catname{Top}$ is also well-defined.

\begin{center}
\begin{tikzcd}
\catname{Graph}_{CP} \arrow[d, "\cutalgebrasymb"'] \arrow[rrd, "\mathcal{D}_E"] &  &               \\
 \catname{Bool}	 \arrow[rr, "\mathrm{Ult}"']                                    &  & \catname{Top}
\end{tikzcd}

\end{center}
We can also define a category $\catname{Graph}_{LP}$ with objects as connected graphs, and morphisms being the strongly line-proper morphisms. Here we can define both $\cutalgebrasymb$ and $L$. $L$ has its image in the category $\catname{Graph}_{SP}$ with objects as connected graphs and morphisms being the strongly proper morphisms - where $\Omega$ is well defined. So we have both compositions $ \mathcal{D}_E =\Omega \circ L = \mathrm{Ult} \circ \cutalgebrasymb: \catname{Graph}_{LP} \rightarrow \catname{Top}$, which means, finally, a proper commuting diagram. 

\begin{center}
\begin{tikzcd}
\catname{Graph}_{LP} \arrow[d, "\cutalgebrasymb"'] \arrow[rrd, "\mathcal{D}_E"] \arrow[rr, "L"] &  & \catname{Graph}_{SP} \arrow[d, "\Omega"] \\
 \catname{Bool}	 \arrow[rr, "\mathrm{Ult}"']                                                    &  & \catname{Top}                           
\end{tikzcd}

\end{center}

    \section{The tangle space functor} 

We could use similar ideas to make the association of a graph $G$ to its tangle space $\Theta(G)$ into a functor. First, we turn $\sepalgebrasymb$ into a functor in the same way did for the cut algebra.

\begin{definition}[Separation-proper morphisms]
    We say a graph homomorphism $\phi:G_1 \to G_2$ is \textbf{separation-proper} when $\phi^{-1}:\sepalgebrasymb(G_2) \to \sepalgebrasymb(G_1)$ is well-defined. These will be precisely the graph homomorphisms that satisfy
    \begin{itemize}
        \item[(i)] If $A \subset E(G_2)$ is connected to its complement $G_2 \setminus A$ via finitely many vertices, then $L(\phi)^{-1}(A)$ is connected to its complement $L(\phi)^{-1}(G_2 \setminus A)$ via finitely many vertices.
        \item[(ii)] For an $e \in E(G_2)$, then $L(\phi)^{-1}(e)$ is a finite set of edges.
    \end{itemize}
\end{definition}

This class of morphism is closed by composition, so it is already good enough to define a category and for us to have a well defined functor $\sepalgebrasymb$, and therefore also its composition $\Theta = \mathrm{Ult} \circ \sepalgebrasymb$. Now, since $\Omega(G)$ is homeomorphically contained in $\Theta(G)$, it is natural to think of the association of a $G$ to $\Theta(G)$ as an extension of the association of $G$ to $\Omega(G)$. However, the strongly proper class of morphisms we used to define the functor $\Omega(G)$ does not match with the class of separation-proper morphisms.

\begin{example}
    \begin{itemize}
        \item[(i)] A strongly proper morphism that is not separation-proper: Take an infinite star $G_1$ with center $c$ and tips $n \in \mathbb{N}$, and $G_2$ a graph with two vertices $a,b$ connected by an edge $e$. Define $\phi: G_1 \rightarrow G_2$ by mapping $c$ to $a$ and every $n$ to $b$. Then $L(\phi)^{-1}(e)$ is infinite, so $\phi$ is not separation-proper. Since $G_1$ has no rays, $\phi$ is strongly proper.
        \item[(ii)] A separation-proper morphism that is not strongly proper: Take $G_1$ to be a ray $<v_1,v_2, v_3, \dots>$ where each $v_i$ is connected to a single other vertex $w_i$ that is not connected to any other vertex. Take $G_2$ to be a star with center $c$ and tips $n \in \mathbb{N}$ and connect with an edge each $n$ to $n +1$. Define $\phi: G_1 \rightarrow G_2$ to be the map that takes each $v_n$ to $n$ and all the $w_i$ to $c$. This map is not strongly proper since $\phi^{-1}(c)$ adheres to a ray. It is separation-proper due to it being injective on the edges and the fact that none of the $w_i$ can connect two different edges.
    \end{itemize}
\end{example}

We can find a nice class of morphisms by copying the strongly line-proper definition.

\begin{definition}[Tangle-proper]
A graph homomorphism $\phi:G_1 \rightarrow G_2$ is called \textbf{tangle-proper} when  $\phi^{-1}(v)$ is finite for every vertex $v \in V(G_2)$.
\end{definition}

\begin{proposition}
    Take a graph homomorphism $\phi:G_1 \rightarrow G_2$. If $\phi$ is tangle-proper, then it is separation-proper and strongly proper.
\end{proposition}
\begin{proof}
It is strongly proper due to finite sets being dispersed. For any $A \subset E(G_2)$, define $F = V(A, E(G_2) \setminus A)$, and then we'll have that $V(L(\phi)^{-1}(A), E(G_1) \setminus L(\phi)^{-1}(A)) \subset \phi^{-1}(F)$. If $F$ is finite, then $\phi^{-1}(F)$ is finite, so this shows $\phi^{-1}$ preserves finite separations. Also, for an $e \in E(G_2)$, $L(\phi)^{-1}(e)$ must be finite because if it wasn't, then one of the vertices incident in $e$ would be mapped into an infinite number of times.
\end{proof}

Of course, the composition of tangle-proper morphisms is still tangle-proper, so we can define a category $\catname{Graph}_{TP}$. In it, $\sepalgebrasymb$ is a well defined functor, and so we can define the functor $\Theta = \mathrm{Ult} \circ \sepalgebrasymb$. We can also define the $\Omega$ functor in this category, as we have described in the previous section.

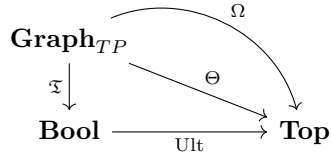
\begin{center}
\begin{tikzcd}
\catname{Graph}_{TP} \arrow[d, "\sepalgebrasymb"'] \arrow[rrd, "\Theta"] \arrow[rrd, "\Omega", bend left=49] &  &               \\
 \catname{Bool}	 \arrow[rr, "\mathrm{Ult}"']                                                                 &  & \catname{Top}
\end{tikzcd}
\end{center}

\section{Final remarks and questions}
We hope our work will provide foundation and inspiration for further work on infinite graph
theory that makes ample use of category theory and powerful algebraic constructions such as Boolean algebras and Stone duality. Here are some questions regarding the constructions we presented here.

\begin{question}
Is there a more natural class of morphisms to be used in defining the $\Theta$ functor?
\end{question}
\begin{question}
What is the class of Boolean algebras that can be characterized as separation algebras or cut algebras of a graph? Equivalently, what is the class of topological spaces that are tangle spaces or edge-direction spaces?
\end{question}
\begin{question}
Is there a similarly algebraic or categorical construction of tangles for finite graphs?
\end{question}
\begin{question}
Can we apply our functors to uncover topological obstructions to the existence of minors or quotients of a graph?
\end{question}

    \section*{Acknowledgments}
    We thank CAPES for the financial support. 
    \bibliographystyle{alpha}
    \bibliography{./bibliography}
    \Addresses
    \typeout{get arXiv to do 4 passes: Label(s) may have changed. Rerun}

\end{document}